\documentclass[12pt]{amsart}%
\usepackage{graphicx}
\usepackage{amscd, color}
\usepackage{amsmath}
\usepackage{amsfonts}
\usepackage{amssymb}%
\providecommand{\U}[1]{\protect\rule{.1in}{.1in}}
\newtheorem{theorem}{Theorem}[section]
\newtheorem{proposition}[theorem]{Proposition}
\newtheorem{corollary}[theorem]{Corollary}

\newtheorem{problem}[theorem]{Problem}
\newtheorem{definition}[theorem]{Definition}
\numberwithin{equation}{section}

\begin{document}
\author{Daniel Pellegrino and Joedson Santos}
\address[D. Pellegrino]{ Departamento de Matem\'{a}tica, Universidade Federal da Paraíba, 58038-310 - Jo\~{a}o Pessoa, Paraíba, Brazil\\
[J. Santos] {Departamento de Matem\'{a}tica, Universidade Federal de Sergipe,
49.500-000 - Itabaiana, Sergipe, Brazil.}}
\title[Absolutely summing multilinear operators: a panorama]{Absolutely summing multilinear operators: a panorama }

%\dedicatory{}

\thanks{Daniel Pellegrino is supported by CNPq (Edital Casadinho) Grant 620108/2008-8}

\begin{abstract}
This paper has a twofold purpose: to present an overview of the theory of
absolutely summing operators and its different generalizations for the
multilinear setting, and to sketch the beginning of a research project related
to an objective search of \textquotedblleft perfect\textquotedblright%
\ multilinear extensions of the ideal of absolutely summing operators. The
final section contains some open problems that may indicate lines for future investigation.

\end{abstract}
\maketitle

\section{Introduction}

Absolutely summing multilinear operators and homogeneous polynomials between
Banach spaces were first conceived by A. Pietsch \cite{PPPP, pp22} in the
eighties. Pietsch%
%TCIMACRO{\U{b4}}%
%BeginExpansion
\'{}%
%EndExpansion
s work and R. Alencar and M.C. Matos' research report \cite{AlencarMatos} are
usually quoted as the precursors of the now well-known nonlinear theory of
absolutely summing operators. In the last decade this topic of investigation
attracted the attention of many authors and various different concepts related
to summability of nonlinear operators were introduced; this line of research,
besides its intrinsic interests, highlighted abstract questions in the
mainstream of the theory of multi-ideals which contributed to the
revitalization of the general interest in questions related to ideals of
polynomials and multilinear operators (see \cite{note, botstudia,
indagationes, muro, cds}).

This paper has a twofold purpose: to summarize/organize some information
constructed in the last years concerning the different multilinear
generalizations of absolutely summing operators; and to sketch a research
project directed to the investigation of the existence of multilinear ideals
(related to the ideal of absolutely summing operators) satisfying a list of
properties which we consider natural. We define the notion of maximal and
minimal ideals satisfying some given properties and obtain existence results,
using Zorn%
%TCIMACRO{\U{b4}}%
%BeginExpansion
\'{}%
%EndExpansion
s Lemma. We also discuss qualitative results, posing some question on the
concrete nature of the maximal and minimal ideals.

None of our goals has the intention to be exhaustive: the overview of the
multilinear theory of absolutely summing operators will be concentrated in
special properties and has no encyclopedic character. Besides, our approach to
the existence of multi-ideals satisfying some given properties is, of course,
focused on those selected properties.

\section{Absolutely summing operators: an overview}

A. Dvoretzky and C. A. Rogers \cite{DR}, in 1950, solved a long standing
problem in Banach Space Theory, by showing that in every infinite-dimensional
Banach space there is an unconditionally convergent series which fails to be
absolutely convergent. This result answers Problem 122 of the Scottish Book
\cite{Mau} (the problem was raised by S. Banach in \cite[page 40]{Banach32}).

This result attracted the interest of A. Grothendieck who, in \cite{Gro1955},
presented a different proof of Dvoretzky-Rogers Theorem. Grothendieck's
\textquotedblleft R\'{e}sum\'{e} de la th\'{e}orie m\'{e}trique des produits
tensoriels topologiques\textquotedblright\ together with his thesis may be
regarded, in some sense, as the birthplace of the theory of operators ideals.

The concept of absolutely $p$-summing linear operators is due to A. Pietsch
\cite{stu} and the notion of $(q,p)$-summing operator is due to B. Mitiagin
and A. Pe\l czy\'{n}ski \cite{MPel}. Another cornerstone in the theory is J.
Lindenstrauss and A. Pe\l czy\'{n}ski's paper \cite{lp}, which translated
Grothendieck's ideas to an universal language and showed the intrinsic beauty
of the theory and richness of possible applications.

From now on the space of all continuous linear operators from a Banach space
$E$ to a Banach space $F$ will be denoted by $\mathcal{L}(E,F).$ Let
\[
\ell_{p}^{\text{weak}}(E):=\left\{  (x_{j})_{j=1}^{\infty}\subset E:\left\Vert
(x_{j})_{j=1}^{\infty}\right\Vert _{w,p}:=\sup_{\varphi\in B_{E^{\ast}}%
}\left(
%TCIMACRO{\dsum \limits_{j=1}^{\infty}}%
%BeginExpansion
{\displaystyle\sum\limits_{j=1}^{\infty}}
%EndExpansion
\left\vert \varphi(x_{j})\right\vert ^{p}\right)  ^{1/p}<\infty\right\}
\]
and%
\[
\ell_{p}(E):=\left\{  (x_{j})_{j=1}^{\infty}\subset E:\left\Vert (x_{j}%
)_{j=1}^{\infty}\right\Vert _{p}:=\left(
%TCIMACRO{\dsum \limits_{j=1}^{\infty}}%
%BeginExpansion
{\displaystyle\sum\limits_{j=1}^{\infty}}
%EndExpansion
\left\Vert x_{j}\right\Vert ^{p}\right)  ^{1/p}<\infty\right\}  .
\]

If $1\leq p\leq q<\infty,$ we say that a continuous linear operator
$u:E\rightarrow F$ is $(q,p)$-summing if $\left(  u(x_{j})\right)
_{j=1}^{\infty}\in\ell_{q}(F)$ whenever $(x_{j})_{j=1}^{\infty}\in\ell
_{p}^{\text{weak}}(E)$.

The class of absolutely $(q,p)$-summing linear operators from $E$ to $F$ will
be represented by $\Pi_{q,p}\left(  E,F\right)  $ and $\Pi_{p}\left(
E,F\right)  $ if $p=q$ (in this case $u\in\Pi_{p}\left(  E,F\right)  $ is said
to be absolutely $p$-summing).

An equivalent formulation asserts that $u:E\rightarrow F$ is $(q,p)$-summing
if there is a constant $C\geq0$ such that%
\[
\left(
%TCIMACRO{\dsum \limits_{j=1}^{\infty}}%
%BeginExpansion
{\displaystyle\sum\limits_{j=1}^{\infty}}
%EndExpansion
\left\Vert u(x_{j})\right\Vert ^{q}\right)  ^{1/q}\leq C\left\Vert
(x_{j})_{j=1}^{\infty}\right\Vert _{w,p}%
\]
for all $(x_{j})_{j=1}^{\infty}\in\ell_{p}^{\text{weak}}(E)$. The above
inequality can also be replaced by: there is a constant $C\geq0$ such that%
\[
\left(
%TCIMACRO{\dsum \limits_{j=1}^{m}}%
%BeginExpansion
{\displaystyle\sum\limits_{j=1}^{m}}
%EndExpansion
\left\Vert u(x_{j})\right\Vert ^{q}\right)  ^{1/q}\leq C\left\Vert
(x_{j})_{j=1}^{m}\right\Vert _{w,p}%
\]
for all $x_{1},...,x_{m}\in E$ and all positive integers $m$.

The infimum of all $C$ that satisfy the above inequalities defines a norm,
denoted by $\pi_{q,p}(u)$ (or $\pi_{p}(u)$ if $p=q$), and $\left(  \Pi
_{q,p}\left(  E,F\right)  ,\pi_{q,p}\right)  $ is a Banach space.

From now on, if $1<p<\infty,$ the conjugate of $p$ is denoted by $p^{\ast},$
i.e., $\frac{1}{p}+\frac{1}{p^{\ast}}=1.$

For a full panorama of the linear theory of absolutely summing operators we
refer to the classical book \cite{djt}. Here we restrict ourselves to five
pillars of the theory: Dvoretzky-Rogers Theorem, Grothendieck's Inequality,
Grothendieck (Lindenstrauss-Pe\l czy\'{n}ski) $\ell_{1}$-$\ell_{2}$ Theorem,
Lindenstrauss-Pe\l czy\'{n}ski Theorem (on the converse of Grothendieck
$\ell_{1}$-$\ell_{2}$ Theorem) and Pietsch Domination Theorem.

Dvoretzky-Rogers Theorem can be stated in the context of absolutely summing
operators as follows:

\begin{theorem}
[Dvoretzky-Rogers Theorem, 1950]If$\ p\geq1,$ then $\Pi_{p}(E;E)=\mathcal{L}%
(E;E)$ if and only if $\dim E<\infty.$
\end{theorem}

In view of the above result it is natural to ask for the existence of some $p$
and infinite-dimensional Banach spaces $E$ and $F$ for which $\Pi
_{p}(E;F)=\mathcal{L}(E;F).$ This question will be answered by Theorem
\ref{kyp} and Theorem \ref{uyy} below.

The fundamental tool of the theory is Grothendieck's Inequality (the
formulation below is due to Lindenstrauss and Pe\l czy\'{n}ski \cite{lp}). We
omit the proof, but several different proofs can be easily found in the literature:

\begin{theorem}
[Grothendieck's Inequality (version of Lindenstrauss and Pe\l czy\'{n}ski),
1968]\label{dggg} There is a positive constant $K_{G}$ so that, for all
Hilbert space $H$, all $n\in\mathbb{N}$, every matrix $\left(  a_{ij}\right)
_{n\times n}$ and any $x_{1},...,x_{n},$ $y_{1},...,y_{n}$ in the unit ball of
$H$, the following inequality holds:%
\begin{equation}
\left\vert \sum_{i,j=1}^{n}a_{ij}\langle x_{i},y_{j}\rangle\right\vert \leq
K_{G}\sup\left\{  \left\vert \sum_{i,j=1}^{n}a_{ij}s_{i}t_{j}\right\vert
:\left\vert s_{i}\right\vert ,\left\vert t_{j}\right\vert \leq1\right\}  .
\label{29.5}%
\end{equation}

\end{theorem}

A consequence of Grothendieck's Inequality is that every continuous linear
operator from $\ell_{1}$ to $\ell_{2}$ is absolutely $1$-summing. This result
was stated by Lindenstrauss-Pe\l czy\'{n}ski \cite{lp} and is, in some sense,
contained in Grothendieck%
%TCIMACRO{\U{b4}}%
%BeginExpansion
\'{}%
%EndExpansion
s R\'{e}sum\'{e}. This type of result is what is now referred to as a
\textquotedblleft coincidence theorem\textquotedblright, i.e., a situation
where there are Banach spaces $E$ and $F$ and real numbers $1\leq p,q<\infty$
so that
\[
\Pi_{q,p}(E,F)=\mathcal{L}(E,F).
\]
The same terminology will be used for multilinear mappings.

We sketch here one of the most elementary proofs of Grothendieck
(Lindenstrauss-Pe\l czy\'{n}ski) Theorem; the crucial role played by
Grothendieck Inequality is easily seen.

\begin{theorem}
[Grothendieck's $\ell_{1}$-$\ell_{2}$ Theorem (version of Lindenstrauss,
Pe\l czy\'{n}ski), 1968]\label{kyp}Every continuous linear operator from
$\ell_{1}$ to $\ell_{2}$ is absolutely $1$-summing.
\end{theorem}

\begin{proof}
Let $\left(  T_{n}\right)  _{n=1}^{\infty}$ be the sequence of the canonical
projections , i.e.,%
\begin{align*}
T_{n}  &  :\ell_{1}\longrightarrow\ell_{1}\\
x  &  =\sum_{i=1}^{\infty}a_{j}e_{j}\mapsto T_{n}\left(  x\right)  =\sum
_{i=1}^{n}a_{j}e_{j}.
\end{align*}

Let $\left(  x_{k}\right)  _{k=1}^{\infty}\in\ell_{1}^{\text{weak}}\left(
\ell_{1}\right)  $ with%
\[
\left\Vert \left(  x_{k}\right)  _{k=1}^{\infty}\right\Vert _{w,1}%
=\sup_{\varphi\in B_{\ell_{1}^{\ast}}}\sum_{k=1}^{\infty}\left\vert
\varphi\left(  x_{k}\right)  \right\vert \leq1.
\]
One can easily verify that
\[
\left\Vert \left(  T_{n}x_{k}\right)  _{k=1}^{\infty}\right\Vert _{w,1}%
=\sup_{\varphi\in B_{\ell_{1}^{\ast}}}\sum_{k=1}^{\infty}\left\vert
\varphi\left(  T_{n}x_{k}\right)  \right\vert \leq1.
\]

Denoting%
\[
x_{k}=\sum_{j=1}^{\infty}a_{jk}e_{j}\text{ and }T_{n}x_{k}=\sum_{j=1}%
^{n}a_{jk}e_{j},
\]
for each $n,k,$ we can verify that for any positive integers $m\leq n$ and
$\left(  s_{j}\right)  _{j=1}^{n},\left(  t_{k}\right)  _{k=1}^{m}\subset
B_{\mathbb{K}}$, we have%
\[
\left\vert \sum_{j=1}^{n}\sum_{k=1}^{m}a_{jk}s_{j}t_{k}\right\vert \leq1.
\]

Now, let $T\in\mathcal{L}\left(  \ell_{1},\ell_{2}\right)  $ and
$m,n\in\mathbb{N}$, with $n\geq m.$ For each $k,$ $1\leq k\leq m,$ from
Hahn-Banach Theorem and Riesz Representation Theorem there is a $y_{k}\in
\ell_{2},$ with $\left\Vert y_{k}\right\Vert _{2}\leq1,$ so that
\[
\left\Vert TT_{n}x_{k}\right\Vert _{2}=\langle TT_{n}x_{k},y_{k}\rangle.
\]
If $m<n$, we can choose $y_{m+1}=\cdots=y_{n}=0.$ Hence
\[
\sum_{k=1}^{m}\left\Vert TT_{n}x_{k}\right\Vert _{2}=\left\vert \sum_{k=1}%
^{m}\sum_{j=1}^{n}a_{jk}\langle Te_{j},y_{k}\rangle\right\vert .
\]
Now Grothendieck's Inequality comes into play:%
\begin{equation}
\sum_{k=1}^{m}\left\Vert TT_{n}x_{k}\right\Vert _{2}\leq K_{G}\left\Vert
T\right\Vert \sup\left\{  \left\vert \sum_{j=1}^{n}\sum_{k=1}^{n}a_{jk}%
s_{j}t_{k}\right\vert :\left\vert s_{j}\right\vert ,\left\vert t_{k}%
\right\vert \leq1\right\}  \leq K_{G}\left\Vert T\right\Vert \label{mmmnnn2}%
\end{equation}
for all $n,m$, with $n\geq m.$ Since
\[
\lim_{n\rightarrow\infty}T_{n}x_{k}=x_{k},
\]
making $n\rightarrow\infty$ em (\ref{mmmnnn2}), we have%
\[
\sum_{k=1}^{m}\left\Vert Tx_{k}\right\Vert _{2}\leq K_{G}\left\Vert
T\right\Vert
\]
and the proof is done.
\end{proof}

The next result is a kind of reciprocal of the Grothendieck Theorem (for a
proof we refer to \cite{lp}):

\begin{theorem}
[Lindenstrauss, Pe\l czy\'{n}ski, 1968]\label{uyy}If $E$ and $F$ are
infinite-dimensional Banach spaces, $E$ has an unconditional Schauder basis
and $\Pi_{1}(E,F)=\mathcal{L}(E,F)$ then $E=\ell_{1}$ and $F$ is a Hilbert space.
\end{theorem}

Another interesting feature of absolutely summing operators is the Domination-Theorem:

\begin{theorem}
[Pietsch-Domination Theorem, 1967]\label{ppk}If $E$ and $F$ are Banach spaces,
a continuous linear operator $T:E\rightarrow F$ is absolutely $p$-summing if
and only if there is a constant $C>0$ and a Borel probability measure $\mu$ on
the closed unit ball of the dual of $E,$ $\left(  B_{E^{\ast}},\sigma(E^{\ast
},E)\right)  ,$ such that%
\begin{equation}
\left\Vert T(x)\right\Vert \leq C\left(  \int_{B_{E^{\ast}}}\left\vert
\varphi(x)\right\vert ^{p}d\mu\right)  ^{\frac{1}{p}}. \label{gupdt}%
\end{equation}

\end{theorem}

\begin{proof}
(Sketch) If (\ref{gupdt}) holds it is easy to show that $T$ is absolutely
$p$-summing. For the converse, consider the (compact) set $P(B_{E^{\ast}})$ of
the probability measures in $C(B_{E^{\ast}})^{\ast}$ (endowed with the
weak-star topology). For each $(x_{j})_{j=1}^{m}$ in $E,$ and $m\in
\mathbb{N},$ let $g:P(B_{E^{\ast}})\rightarrow\mathbb{R}$ be defined by%
\[
g\left(  \rho\right)  =\sum_{j=1}^{m}\left[  \left\Vert T(x_{j})\right\Vert
^{p}-C^{p}\int_{B_{E^{\ast}}}\left\vert \varphi(x_{j})\right\vert ^{p}%
d\rho(\varphi)\right]
\]
and $\mathcal{F}$ be the set of all such $g.$ It is not difficult to show that
$\mathcal{F}$ is concave and each $g\in\mathcal{F}$ is continuous and \ convex.

Besides, for each $g\in\mathcal{F}$ there is a measure $\mu_{g}\in
P(B_{E^{\ast}})$ such that $g(\mu_{g})\leq0.$ In fact, from the compactness of
$B_{E^{\ast}}$ and Weierstrass' theorem there is a $\varphi_{0}\in K$ so that%
\[
\sum_{j=1}^{m}\left\vert \varphi_{0}(x_{j})\right\vert ^{p}=\sup_{\varphi\in
B_{E^{\ast}}}\sum_{j=1}^{m}\left\vert \varphi(x_{j})\right\vert ^{p}.
\]
Then, considering the Dirac measure $\mu_{g}=\delta_{\varphi_{0}},$ we deduce
$g(\mu_{g})\leq0.$ So, Ky Fan Lemma (see \cite[page 40]{mono}) ensures that
there exists a $\mu\in P(B_{E^{\ast}})$ so that%
\[
g(\mu)\leq0
\]
for all $g\in\mathcal{F}$ and by choosing an arbitrary $g$ with $m=1$ the
proof is done.
\end{proof}

Using the canonical inclusions from $L_{p}$ spaces we get the following result:

\begin{corollary}
[Inclusion Theorem]If $1\leq r\leq s<\infty,$ then every absolutely
$r$-summing operator is absolutely $s$-summing.
\end{corollary}

The 70's witnessed the emergence of the notion of cotype of a Banach
space, with contributions from J. Hoffmann-J\o rgensen \cite{HJ}, B. Maurey
\cite{Ma2}, S. Kwapie\'{n} \cite{K22}, E. Dubinsky, A. Pe\l czy\'{n}ski and H. P. Rosenthal \cite{DPR} among others; in 1976 the strong connection between the notions of cotype and absolutely summing operators became evident with the work of B. Maurey and G. Pisier
\cite{pisier}. Let us recall the notion of cotype.

The Rademacher functions
\[
r_{n}:\left[  0,1\right]  \longrightarrow\mathbb{R},n\in\mathbb{N}%
\]
are defined as%
\[
r_{n}\left(  t\right)  :=sign\left(  \sin2^{n}\pi t\right)  .
\]

A Banach space $E$ is said to have cotype $q\geq2$ if there is a constant
$K\geq0$ so that, for all positive integer $n$ and all $x_{1},...,x_{n}$ in
$E$, we have%
\begin{equation}
\left(  \underset{i=1}{\overset{n}{%
%TCIMACRO{\dsum }%
%BeginExpansion
{\displaystyle\sum}
%EndExpansion
}}\left\Vert x_{i}\right\Vert ^{q}\right)  ^{1/q}\leq K\left(  \int_{0}%
^{1}\left\Vert \underset{i=1}{\overset{n}{%
%TCIMACRO{\dsum }%
%BeginExpansion
{\displaystyle\sum}
%EndExpansion
}}r_{i}\left(  t\right)  x_{i}\right\Vert ^{2}dt\right)  ^{1/2}\text{.}%
\label{2.3}%
\end{equation}
We denote by $C_{q}\left(  E\right)  $ the infimum of all such $K$ satisfying
$\left(  \ref{2.3}\right)  $ and $\cot E$ denotes the infimum of the cotypes
assumed by $E$, i.e.,%
\[
\cot E=\inf\left\{  2\leq q\leq\infty;E\text{ has cotype }q\right\}  .
\]
It is worth mentioning that $E$ need not to have cotype $\cot E.$

The following combination of results of Maurey, Pisier \cite{pisier} and
Talagrand \cite{T1} are self-explanatory:

\begin{theorem}
[Maurey, Pisier, 1976 + Talagrand, 1992]If a Banach space $E$ has finite
cotype $q$, then $id_{E}$ is absolutely $(q,1)$-summing. The converse is true,
except for $q=2$.
\end{theorem}

\begin{proof}
(Easy part) If $E$ has cotype $q<\infty,$ then
\begin{align*}
\left(  \underset{i=1}{\overset{n}{%
%TCIMACRO{\dsum }%
%BeginExpansion
{\displaystyle\sum}
%EndExpansion
}}\left\Vert x_{i}\right\Vert ^{q}\right)  ^{1/q}  &  \leq C_{q}(E)\left(
\int_{0}^{1}\left\Vert \underset{i=1}{\overset{n}{%
%TCIMACRO{\dsum }%
%BeginExpansion
{\displaystyle\sum}
%EndExpansion
}}r_{i}\left(  t\right)  x_{i}\right\Vert ^{2}dt\right)  ^{1/2}\\
&  \leq C_{q}(E)\underset{\left\vert t\right\vert \leq1}{\sup}\left\Vert
\sum\limits_{j=1}^{n}r_{j}(t)x_{j}\right\Vert \\
&  \leq C_{q}(E)\left\Vert (x_{j})_{j=1}^{n}\right\Vert _{w,1}.
\end{align*}
The rest of the proof is quite delicate.
\end{proof}

In the 80's the part of the focus of the investigation related to absolutely
summing operators was naturally moved to the nonlinear setting, which will be
treated in the next sections. However the linear theory is still alive and
there are still interesting problems being investigated (see, for example,
\cite{de3, ko4}). For recent results we mention \cite{belg, PellZ, ku, ku2}:

Recent results reinforce the important role played by cotype:

\begin{theorem}
[Botelho, Pellegrino, 2009 ]\label{oik}(\cite{belg, PellZ}) Let $E$ and $F$ be
infinite-dimensional Banach spaces.

(i) If $\Pi_{1}(E,F)=\mathcal{L}(E,F)$ then $\cot E=\cot F=2$.

(ii) If $2\leq r<\cot F$ and $\Pi_{q,r}(E,F)=\mathcal{L}(E,F),$ then
$\mathcal{L}(\ell_{1},\ell_{\cot F})=\Pi_{q,r}(\ell_{1},\ell_{\cot F})$.

(iii) If $\cot F=\infty$ and $p\geq1$, there exists a continuous linear
operator from $E$ to $F$ which fails to be $p$-summing.
\end{theorem}

In a completely different direction, recent papers have investigated linear
absolutely summing operators in the context of the theory of
lineability/spaceability (see \cite{bd, seo, timoney}). For example, in
\cite{timoney} the following result (which can be interpreted as a
generalization of results from \cite{davis}) is shown:

\begin{theorem}
[Kitson, Timoney, 2010]Let $\mathcal{K}(E,F)$ denote the space of compact
linear operators from $E$ to $F$. If $E$ and $F$ are infinite-dimensional
Banach spaces and $E$ is super-reflexive, then
\[
A=\mathcal{K}(E,F)\smallsetminus%
%TCIMACRO{\dbigcup \limits_{1\leq p<\infty}}%
%BeginExpansion
{\displaystyle\bigcup\limits_{1\leq p<\infty}}
%EndExpansion
\Pi_{p}(E,F)
\]
is spaceable (i.e., $A\cup\{0\}$ contains a closed infinite-dimensional vector space).
\end{theorem}

\section{Operator ideals and multi-ideals: generating multi-ideals}

The theory of operator ideals is due to Pietsch and goes back to his monograph
\cite{mono} in 1978. An operator ideal $\mathcal{I}$ is a subclass of the
class $\mathcal{L}$ of all continuous linear operators between Banach spaces
such that for all Banach spaces $E$ and $F$ its components%
\[
\mathcal{I}(E;F):=\mathcal{L}(E;F)\cap\mathcal{I}%
\]
satisfy:

(1) $\mathcal{I}(E;F)$ is a linear subspace of $\mathcal{L}(E;F)$ which
contains the finite rank operators.

(2) (Ideal property) If $u\in\mathcal{I}(E;F)$, $v\in\mathcal{L}(G;E)$ for
$j=1,\ldots,n$ and $t\in\mathcal{L}(F;H)$, then $t\circ u\circ v\in
\mathcal{I}(G;H)$.

The structure of operator ideals is shared by the most important classes of
operators that appear in Functional Analysis, such as compact, weakly compact,
nuclear, approximable, absolutely summing, strictly singular operators, among
many others.

The multilinear theory of operator ideals was also sketched by Pietsch in
\cite{PPPP}.

From now on $\mathbb{K}$ represents the field of all scalars (complex or
real), and $\mathbb{N}$ denotes the set of all positive integers. For
$n\geq1,$ the Banach space of all continuous $n$-linear mappings from
$E_{1}\times\cdots\times E_{n}$ into $F$ endowed with the $\sup$ norm is
denoted by $\mathcal{L}(E_{1},...,E_{n};F).$

An ideal of multilinear mappings (or multi-ideal) $\mathcal{M}$ is a subclass
of the class of all continuous multilinear operators between Banach spaces
such that for a positive integer $n$, Banach spaces $E_{1},\ldots,E_{n}$ and
$F$, the components
\[
\mathcal{M}(E_{1},\ldots,E_{n};F):=\mathcal{L}(E_{1},\ldots,E_{n}%
;F)\cap\mathcal{M}%
\]
satisfy:

\bigskip

(i) $\mathcal{M}(E_{1},\ldots,E_{n};F)$ is a linear subspace of $\mathcal{L}%
(E_{1},\ldots,E_{n};F)$ which contains the $n$-linear mappings of finite type.

(ii) The ideal property: if $A\in\mathcal{M}(E_{1},\ldots,E_{n};F)$, $u_{j}%
\in\mathcal{L}(G_{j};E_{j})$ for $j=1,\ldots,n$ and $t\in\mathcal{L}(F;H)$,
then $t\circ A\circ(u_{1},\ldots,u_{n})$ belongs to $\mathcal{M}(G_{1}%
,\ldots,G_{n};H)$.

Moreover, there is a function $\Vert\cdot\Vert_{\mathcal{M}}\colon
\mathcal{M}\longrightarrow\lbrack0,\infty)$ satisfying

\bigskip

(i') $\Vert\cdot\Vert_{\mathcal{M}}$ restricted to $\mathcal{M}(E_{1}%
,\ldots,E_{n};F)$ is a norm, for all Banach spaces $E_{1},\ldots,E_{n}$ and
$F$, which makes $\mathcal{M}(E_{1},\ldots,E_{n};F)$ a Banach space.

(ii') $\Vert A\colon\mathbb{K}^{n}\longrightarrow\mathbb{K}:A(\lambda
_{1},\ldots,\lambda_{n})=\lambda_{1}\cdots\lambda_{n}\Vert_{\mathcal{M}}=1$
for all $n$,

(iii') If $A\in\mathcal{M}(E_{1},\ldots,E_{n};F)$, $u_{j}\in\mathcal{L}%
(G_{j};E_{j})$ for $j=1,\ldots,n$ and $v\in\mathcal{L}(F;H)$, then $\Vert
v\circ A\circ(u_{1},\ldots,u_{n})\Vert_{\mathcal{M}}\leq\Vert v\Vert\Vert
A\Vert_{\mathcal{M}}\Vert u_{1}\Vert\cdots\Vert u_{n}\Vert$.\bigskip

However, the construction of adequate multilinear and polynomial extensions of
a given operator ideal needs some care. The first is that, given positive
integers $n_{1}$ and $n_{2}$, the respective levels of $n_{1}$-linearity and
$n_{2}$-linearity need to have some inter-connection and obviously a strong
relation with the original level $(n=1)$. This pertinent preoccupation has
appeared in different recent papers, with the notions of ideals closed for
scalar multiplication, closed for differentiation and the notions of coherent
and compatible multilinear ideals (see \cite{botstudia, indagationes, muro}).

The following properties illustrate the essence of the aforementioned
inter-connection between the levels of the multi-ideal (these concepts are
natural adaptations from the analogous for polynomials defined in
\cite{indagationes}).

\begin{definition}
[cud multi-ideal]An ideal of multilinear mappings $\mathcal{M}$ is closed
under differentiation (cud)\textbf{ }if, for all $n$, $E_{1},...,E_{n},F$ and
$T\in\mathcal{M}(E_{1},...,E_{n};F)$, every linear operator obtained by fixing
$n-1$ vectors $a_{1},...,a_{j-1},a_{j+1},...,a_{n}$ belongs to $\mathcal{M}%
(E_{j};F)$ for all $j=1,...,n.$
\end{definition}

\begin{definition}
[csm multi-ideal]An ideal of multilinear mappings $\mathcal{M}$ is closed for
scalar multiplication (csm) if for all $n$, $E_{1},...,E_{n},E_{n+1},F,$
$T\in\mathcal{M}(E_{1},...,E_{n};F)$ and $\varphi\in E_{n+1}^{\ast}$, the map
$\varphi T$ belongs to $\mathcal{M}(E_{1},...,E_{n},E_{n+1};F)$.
\end{definition}

For the theory of polynomials and multilinear mappings between Banach spaces
we refer to \cite{Di, Mu}.

\section{Multiple summing multilinear operators: the prized idea}

Few know that the concept of multiple $p$-summing mappings was introduced in a
research report of M.C. Matos in 1992 \cite{rr}, under the terminology of
\textquotedblleft strictly absolutely summing multilinear
mappings\textquotedblright. The motivation of Matos was a question of Pietsch
on the eventual coincidence of the Hilbert-Schmidt $n$-linear functionals and
the space of absolutely $(s;r_{1},...,r_{n})$-summing $n$-linear functionals
for some values of $s$ and $r_{k},k=1,...,n.$ In this research report, the
first properties of this class are introduced, as well as the connections with
Hilbert-Schmidt multilinear operators and a solution to Pietsch's question in
the context of strictly absolutely summing multilinear mappings.

However, this research report was not published and only in 2003 Matos
\cite{collec} published an improved version of this preprint, now using the
terminology of \textit{fully summing multilinear mappings.} At the same time,
and independently, Bombal, P\'{e}rez-Garc\'{\i}a and Villanueva \cite{bombal,
jmaa} introduced the same concept, under the terminology of multiple summing
multilinear operators.

Since then this class has gained special attention, being considered by
several authors as the most important multilinear generalization of the ideal
of absolutely summing operators. For this reason we will dedicate more
attention to the description of this class.

A fair description of the subject should begin in 1930, when Littlewood
\cite{LLL} (see \cite{bla} for a recent approach) proved his Littlewood's
$4/3$ inequality asserting that
\[
\left(  \sum\limits_{i,j=1}^{N}\left\vert U(e_{i},e_{j})\right\vert ^{\frac
{4}{3}}\right)  ^{\frac{3}{4}}\leq\sqrt{2}\left\Vert U\right\Vert
\]
for every bilinear form $U:\ell_{\infty}^{N}\times\ell_{\infty}^{N}%
\rightarrow\mathbb{C}$ and every positive integer $N.$ One year later
Bohnenblust and Hille \cite{BH} (see also \cite{sevilla, Def2}) improved this
result to multilinear forms by showing that for every positive integer $n$
there is a $C_{n}>0$ so that
\begin{equation}
\left(  \sum\limits_{i_{1},...,i_{n}=1}^{N}\left\vert U(e_{i_{^{1}}%
},...,e_{i_{n}})\right\vert ^{\frac{2n}{n+1}}\right)  ^{\frac{n+1}{2n}}\leq
C_{n}\left\Vert U\right\Vert \label{ju}%
\end{equation}
for every $n$-linear mapping $U:\ell_{\infty}^{N}\times\cdots\times
\ell_{\infty}^{N}\rightarrow\mathbb{C}$ and every positive integer $N$.

Using that $\mathcal{L}\left(  c_{0};E\right)  $ is isometrically isomorphic
to $\ell_{1}^{w}\left(  E\right)  $ (see \cite{djt}), Bohnenblust-Hille
inequality can be re-written as (details can be found in \cite{pgtese}):

\begin{theorem}
[Bohnenblust-Hille, re-written (P\'{e}rez-Garc\'{\i}a, 2003)]\label{ytr}If
$1\leq p<\infty$, and $n$ is a positive integer and $E_{1},...,E_{n}$ are
Banach spaces and $U\in\mathcal{L}(E_{1},\ldots,E_{n};\mathbb{K}),$ then there
exists a constant $C_{n}\geq0$ such that%
\begin{equation}
\left(  \sum_{j_{1},\ldots,j_{n}=1}^{N}\left\vert U(x_{j_{1}}^{(1)}%
,\ldots,x_{j_{n}}^{(n)})\right\vert ^{\frac{2n}{n+1}}\right)  ^{\frac{n+1}%
{2n}}\leq C_{n}\prod_{k=1}^{n}\left\Vert (x_{j}^{(k)})_{j=1}^{N}\right\Vert
_{w,1} \label{juo}%
\end{equation}
for every positive integer $N$ and $x_{j}^{(k)}\in E_{k}$, $k=1,...,n$ and
$j=1,...,N.$
\end{theorem}

In this sense Bohnenblust-Hille theorem can be interpreted as the beginning of
the notion of multiple summing operators:

If $1\leq p_{1},...,p_{n}\leq q<\infty,$ $T:E_{1}\times\cdots\times
E_{n}\rightarrow F$ is multiple\emph{ }$(q;p_{1},...,p_{n})$-summing
($T\in\mathcal{L}_{m,(q,p_{1},...,p_{n})}(E_{1},...,E_{n};F)$) if there exists
$C_{n}>0$ such that
\begin{equation}
\left(  \sum_{j_{1},...,j_{n}=1}^{\infty}\Vert T(x_{j_{1}}^{(1)},...,x_{j_{n}%
}^{(n)})\Vert^{q}\right)  ^{1/q}\leq C_{n}\prod\limits_{k=1}^{n}\Vert
(x_{j}^{(k)})_{j=1}^{\infty}\Vert_{w,p_{k}}\text{ } \label{jup2}%
\end{equation}
for every $(x_{j}^{(k)})_{j=1}^{\infty}\in\ell_{p_{k}}^{w}(E_{k})$,
$k=1,...,n$.

When $p_{1}=...=p_{n}=p$ we write $\mathcal{L}_{m,(q;p)}$ instead of
$\mathcal{L}_{m,(q;p_{1},...,p_{n})};$ when $p_{1}=...=p_{n}=p=q$ we write
$\mathcal{L}_{m,p}$ instead of $\mathcal{L}_{m,(q;p_{1},...,p_{n})}.$ The
infimum of the constants $C_{n}$ satisfying (\ref{jup2}) defines a norm in
$\mathcal{L}_{m,(q,p)}$ and is denoted by $\pi_{q;p_{1},...,p_{k}}$ (or
$\pi_{q;p}$ if $p_{1}=\cdots=p_{k}=p$ or even $\pi_{p}$ when $p_{1}%
=\cdots=p_{k}=p=q$). It is worth mentioning that the essence of the notion of
multiple summing multilinear operators, for bilinear operators, also appears
in the paper of Ramanujan and Schock \cite{Ram}.

It is well-known that the power $\frac{2n}{n+1}$ in Bonenblust-Hille Theorem
\ref{ytr} is optimal. The constant $C_{n}$ from (\ref{juo}) is the same
constant from (\ref{ju}). The optimal values are not known. For recent
estimates for $C_{n}$ we refer to \cite{psseo}. For example, in the real case,
for $2\leq n\leq14,$ in \cite{psseo} it is shown that $C_{n}\leq$
$2^{\frac{n^{2}+6n-8}{8n}}$ if $n$ is even and by $C_{n}\leq2^{\frac
{n^{2}+6n-7}{8n}}$ if $n$ is odd (these estimates are derived from
\cite{Def2}). In the complex case, H. Qu\'{e}ffelec \cite{Que}, A. Defant and
P. Sevilla-Peris \cite{sevilla} have proved that $C_{n}\leq\left(  \frac
{2}{\sqrt{\pi}}\right)  ^{n-1}$ but for $n\geq8$ better estimates can be also
found in \cite{psseo} (also derived from \cite{Def2}).

So, since the power $\frac{2n}{n+1}$ is sharp, one might not expect that the
class of multiple summing operators shall lift the trivial coincidence
situations from the linear case, i.e.,%
\[
\Pi_{p}(E;\mathbb{K})=\mathcal{L}(E,\mathbb{K})
\]
for every Banach spaces $E$ but, in general,%
\[
\mathcal{L}_{m,p}(^{n}E;\mathbb{K})\neq\mathcal{L}(^{n}E,\mathbb{K}).
\]

The multi-ideal of multiple summing multilinear operators is, by far, the most
investigated class related to the multilinear theory of absolutely summing
operators (see \cite{botp, andreas david, sevilla, davidstudia} and references
therein). The reason for the success of this generalization of absolutely
summing operators is perhaps the nice combination of nontrivial good
properties, as coincidence theorems similar to those from the linear theory
(\cite{bombal, botpams, botp}), and challenging problems as the inclusion
theorem which holds in very special situations.

The main results below are presented with the respective dates. In the case of
results that appeared in a thesis or dissertation and were published after, we
have chosen the date of the thesis/dissertation.

A first remark on the class of multiple summing multilinear operators is that
it is easy to show that coincidence results for multiple summing multilinear
operators always imply in the respective linear ones (details can be found in
\cite{spp}):

\begin{proposition}
If $\mathcal{L}(E_{1},\ldots,E_{n};F)=\mathcal{L}_{m,(q;p_{1},\ldots,p_{n}%
)}(E_{1},\ldots,E_{n};F)$, then
\[
\mathcal{L}(E_{j};F)=\Pi_{q,p_{j}}(E_{j};F),j=1,\ldots,n.
\]

\end{proposition}

Bohnenblust-Hille type results were also studied in a different perspective
(trying to replace $\frac{2n}{n+1}$ by $2$ by changing the $1$-weak norm by
some $p_{n}$-weak norm). If $(p_{k})_{k=0}^{\infty}$ is the sequence of real
numbers given by
\[
p_{0}=2\mbox{ and }p_{k+1}=\frac{2p_{k}}{1+p_{k}}\mbox{ for }k\geq0,
\]
then the following Bohnenblust-Hille type result is valid:

\begin{theorem}
[Botelho, Braunss, Junek, Pellegrino, 2009](\cite{botpams}) Let $E_{1}%
,\ldots,E_{n}$ be Banach spaces of cotype $2$. If $k$ is the natural number
such that $2^{k-1}<n\leq2^{k}$, then
\[
\mathcal{L}(E_{1},\ldots,E_{n};\mathbb{K})=\mathcal{L}_{m(2;p_{k},\ldots
,p_{k})}(E_{1},\ldots,E_{n};\mathbb{K}).
\]

\end{theorem}

A very important contribution to the theory of multiple summing multilinear
operators was given in D. P\'{e}rez-Garc\'{\i}a's thesis, where several new
results and techniques are presented, inspiring several related papers. The
inclusion theorems proved by P\'{e}rez-Garc\'{\i}a deserves special attention:

\begin{theorem}
[P\'{e}rez-Garc\'{\i}a, 2003](\cite{pgtese, davidstudia})\label{bbn} If $1\leq
p\leq q<2$, then $\mathcal{L}_{m,p}(E_{1},...,E_{n};F)\subset\mathcal{L}%
_{m,q}(E_{1},...,E_{n};F).$
\end{theorem}

P\'{e}rez-Garc\'{\i}a has also shown that the above result cannot be extended
in the sense that for each $q>2$ there exists $T\in\mathcal{L}_{m,p}(^{2}%
\ell_{1};\mathbb{K})$ for $1\leq p\leq2$ which does not belong to
$\mathcal{L}_{m,q}(^{2}\ell_{1};\mathbb{K}).$

When the $F$ has cotype $2$ the result is slight better:

\begin{theorem}
[P\'{e}rez-Garc\'{\i}a, 2003](\cite{pgtese, davidstudia})\label{unv} If $1\leq
p\leq q\leq2$ and $F$ has cotype $2$, then $\mathcal{L}_{m,p}(E_{1}%
,...,E_{n};F)\subset\mathcal{L}_{m,q}(E_{1},...,E_{n};F).$
\end{theorem}

When the spaces from the domain have cotype $2$, the inclusions from Theorem
\ref{bbn} become coincidences (for a simple proof we refer to \cite{michels,
enama}; the main tool used in the proof are results from \cite{arregui}):

\begin{theorem}
[Botelho, Pellegrino, 2008 and Popa, 2009](\cite{enama, popa}) If $1\leq
p,q<2$ and $E_{1},...,E_{n}$ have cotype $2$, then%
\[
\mathcal{L}_{m,p}(E_{1},...,E_{n};F)=\mathcal{L}_{m,q}(E_{1},...,E_{n};F).
\]

\end{theorem}

Recently, in \cite{michels}, it was shown (using complex interpolation and an
argument of complexification) that a more general version of Theorem \ref{unv}
is valid when the spaces from the domain are $\mathcal{L}_{\infty}$-spaces:

\begin{theorem}
[Botelho, Michels, Pellegrino, 2010]Let $1\leq p\leq q\leq\infty$ and
$E_{1},\ldots,E_{n}$ {be} $\mathcal{L}_{\infty}$-spaces. Then $\mathcal{L}%
_{m,p}(E_{1},\ldots,E_{n};F)\subset\mathcal{L}_{m,q}(E_{1},\ldots,E_{n};F)$.
\end{theorem}

The proofs of the above results are technical and we omit them. For other
related results we mention (\cite{botpams, michels, davidstudia}).

Coincidence theorems are also a fruitful subject in the context of multiple
summing operators. For example, D. P\'{e}rez-Garc\'{\i}a proved that
Grothendieck's Theorem is valid for multiple summing multilinear operators:

\begin{theorem}
[P\'{e}rez-Garc\'{\i}a, 2003]\cite{pgtese} If $1\leq p\leq2$, then
$\mathcal{L}_{m,p}(^{n}\ell_{1};\ell_{2})=\mathcal{L}(^{n}\ell_{1};\ell_{2}).$
\end{theorem}

We sketch the proof of a more general result from \cite{botp}, which is
inspired in P\'{e}rez-Garc\'{\i}a's ideas:

\begin{theorem}
[Botelho, Pellegrino, 2009]\label{novoteorema} Let $r\geq s\geq1$. If
$\mathcal{L}(\ell_{1};F)=\Pi_{r;s}(\ell_{1};F)$, then
\[
\mathcal{L}(^{n}\ell_{1};F)=\mathcal{L}_{m,(r;\,\min\{s,2\})}(^{n}\ell_{1};F)
\]
for every $n\in\mathbb{N}$.
\end{theorem}

\begin{proof}
(Sketch) In \cite[Theorem 3.4]{jmaa} it is shown that when $1\leq p\leq2,$
then $\mathcal{L}_{m,p}(^{2}\ell_{1};\mathbb{K})=\mathcal{L}(^{2}\ell
_{1};\mathbb{K}),$ and
\begin{equation}
\pi_{p}(\cdot)\leq K_{G}^{2}\Vert\cdot\Vert. \label{ljg}%
\end{equation}

Let $(x_{j}^{(1)})_{j=1}^{m_{1}},\ldots,(x_{j}^{(n)})_{j=1}^{m_{n}}$ be $n$
finite sequences in $\ell_{1}$. Using that $\widehat{\otimes}_{\pi}^{k}%
\ell_{1}$ is isometrically isomorphic to $\ell_{1}$ and (\ref{ljg}), one can
prove that, for every $1\leq p\leq2$,
\[
\left\Vert (x_{j_{1}}^{(1)}\otimes\cdots\otimes x_{j_{n}}^{(n)})_{j_{1}%
,\ldots,j_{n}=1}^{m_{1},\ldots,m_{n}}\right\Vert _{w,p}\leq K_{G}%
^{2n-2}\left\Vert (x_{j}^{(1)})_{j=1}^{m_{1}}\right\Vert _{w,p}\cdots
\left\Vert (x_{j}^{(n)})_{j=1}^{m_{n}}\right\Vert _{w,p}%
\]
\indent Let $A\in\mathcal{L}(^{n}\ell_{1};F)$. By $A_{L}$ we mean the
linearization of $A$ on $\widehat{\otimes}_{\pi}^{n}\ell_{1}$, that is
$A_{L}\in\mathcal{L}(\widehat{\otimes}_{\pi}^{n}\ell_{1};F)$ and $A_{L}%
(x_{1}\otimes\cdots\otimes x_{n})=A(x_{1},\ldots,x_{n})$ for every $x_{j}%
\in\ell_{1}$. Since $\widehat{\otimes}_{\pi}^{n}\ell_{1}$ is isometrically
isomorphic to $\ell_{1}$, by assumption we have that $A_{L}$ is $(r;s)$%
-summing and $\pi_{r;s}(A_{L})\leq M\Vert A_{L}\Vert=M\Vert A\Vert$, where $M$
is a constant independent of $A$. Using the claim with $p=\min\{s,2\}$ we get
\begin{align*}
\left(  \sum_{j_{1},\ldots,j_{n}=1}^{m_{1},\ldots,m_{n}}\left\Vert A(x_{j_{1}%
}^{(1)},\ldots,x_{j_{n}}^{(n)})\right\Vert ^{r}\right)  ^{\frac{1}{r}}  &
\leq\left(  \sum_{j_{1},\ldots,j_{n}=1}^{m_{1},\ldots,m_{n}}\left\Vert
A_{L}(x_{j_{1}}^{(1)}\otimes\cdots\otimes x_{j_{n}}^{(n)})\right\Vert
^{r}\right)  ^{\frac{1}{r}}\\
&  \leq\pi_{r;s}(A_{L})\left\Vert (x_{j_{1}}^{(1)}\otimes\cdots\otimes
x_{j_{n}}^{(n)})_{j_{1},\ldots,j_{n}=1}^{m_{1},\ldots,m_{n}}\right\Vert
_{w,s}\\
&  \leq\pi_{r;s}(A_{L})\left\Vert (x_{j_{1}}^{(1)}\otimes\cdots\otimes
x_{j_{n}}^{(n)})_{j_{1},\ldots,j_{n}=1}^{m_{1},\ldots,m_{n}}\right\Vert
_{w,\min\{s,2\}}\\
&  \leq M\Vert A\Vert K_{G}^{2n-2}\left\Vert (x_{j}^{(1)})_{j=1}^{m_{1}%
}\right\Vert _{w,\min\{s,2\}}\cdots\left\Vert (x_{j}^{(n)})_{j=1}^{m_{n}%
}\right\Vert _{w,\min\{s,2\}},
\end{align*}
which shows that $A$ is multiple $(r;\min\{s,2\})$-summing.\bigskip
\end{proof}

\begin{corollary}
[Botelho, Pellegrino, 2009](\cite{botp}) Let $1\leq p\leq2$, $r\geq p$ and let
$F$ be a Banach space. The following assertions are equivalent:\newline%
\textrm{(a)} $\mathcal{L}(\ell_{1};F)=\mathcal{L}_{m(r;p)}(^{n}\ell_{1}%
;F)$.\newline\textrm{(b)} $\mathcal{L}(^{n}\ell_{1};F)=\mathcal{L}%
_{m(r;p)}(^{n}\ell_{1};F)$ for every $n\in\mathbb{N}$.\newline\textrm{(c)}
$\mathcal{L}(^{n}\ell_{1};F)=\mathcal{L}_{m(r;p)}(^{n}\ell_{1};F)$ for some
$n\in\mathbb{N}$.
\end{corollary}

The connection between linear coincidence results with coincidence results for
multiple summing multilinear operators in indeed stronger:

\begin{theorem}
[Botelho, Pellegrino, 2009](\cite{botp})\label{teorema} Let $p,r\in
\lbrack1,q]$ and let $F$ be a Banach space. Let $B(p,q,r,F)$ denote the set of
all Banach spaces $E$ such that
\[
\mathcal{L}(E;F)=\Pi_{q;p}(E;F)\mathit{~and~}\mathcal{L}(E;\ell_{q}%
(F))=\Pi_{q;r}(E;\ell_{q}(F)).
\]
Then, for every $n\geq2$,
\[
\mathcal{L}(E_{1},\ldots,E_{n};F)=\mathcal{L}_{m(q;r,\ldots,r,p)}(E_{1}%
,\ldots,E_{n};F)
\]
whenever $E_{1},\ldots,E_{n}\in B(p,q,r,F)$.
\end{theorem}

\begin{proof}
(Sketch) Induction on $n$. For the case $n=2,$ let $E_{1},E_{2}\in
B(p,q,r,F)$. By the Open Mapping Theorem there are constants $C_{1}$ and
$C_{2}$ such that
\[
\pi_{q;p}(u)\leq C_{1}\Vert u\Vert\mathrm{~for~every~}u\in\mathcal{L}%
(E_{2};F)\mathit{~and~}%
\]%
\[
\pi_{q;r}(v)\leq C_{2}\Vert v\Vert\mathrm{~for~every~}v\in\mathcal{L}%
(E_{1};\ell_{q}(F)).
\]
Let $A\in\mathcal{L}(E_{1},E_{2};F)$. Given two sequences $(x_{j}^{(1)}%
)_{j=1}^{\infty}\in\ell_{r}^{w}(E_{1})$ and $(x_{j}^{(2)})_{j=1}^{\infty}%
\in\ell_{p}^{w}(E_{2})$, fix $m\in\mathbb{N}$ and consider the continuous
linear operator
\[
A_{1}^{(m)}\colon E_{1}\longrightarrow\ell_{q}(F)~:~A_{1}^{(m)}(x)=(A(x,x_{1}%
^{(2)}),\ldots,A(x,x_{m}^{(2)}),0,0,\ldots).
\]
So, $A_{1}^{(m)}$ is $(q;r)$-summing and $\pi_{q;r}(A_{1}^{(m)})\leq
C_{2}\Vert A_{1}^{(m)}\Vert$. For each $x\in B_{E_{1}}$, consider the
continuous linear operator
\[
A_{x}\colon E_{2}\longrightarrow F~:~A_{x}(y)=A(x,y).
\]
So, $A_{x}$ is $(q;p)$-summing and $\pi_{q;p}(A_{x})\leq C_{1}\Vert A_{x}%
\Vert\leq C_{1}\Vert A\Vert\Vert x\Vert\leq C_{1}\Vert A\Vert$ and we can
obtain
\begin{equation}
\left(  \sum_{j=1}^{m}\sum_{k=1}^{m}\left\Vert A(x_{j}^{(1)},x_{k}%
^{(2)})\right\Vert ^{q}\right)  ^{\frac{1}{q}}\!\leq C_{1}C_{2}\Vert
A\Vert\left\Vert (x_{j}^{(1)})_{j=1}^{m}\right\Vert _{w,r}\left\Vert
(x_{k}^{(2)})_{j=1}^{m}\right\Vert _{w,p},\nonumber
\end{equation}
which shows that $A$ is multiple $(q;r,p)$-summing and $\pi_{q;r,p}(A)\leq
C_{1}C_{2}\Vert A\Vert$.\newline\indent Suppose now that the result holds for
$n$, that is: for every $E_{1},\ldots,E_{n}\in B(p,q,r,F)$, $\mathcal{L}%
(E_{1},\ldots,E_{n};F)=\mathcal{L}_{m(q;r,\ldots,r,p)}(E_{1},\ldots,E_{n};F)$.
To prove the case $n+1$, let $E_{1},\ldots,E_{n+1}\in B(p,q,r,F)$. Since
$E_{2},\ldots,E_{n+1}$ belong to $B(p,q,r,F)$, we have $\mathcal{L}%
(E_{2},\ldots,E_{n+1};F)=\mathcal{L}_{m(q;r,\ldots,r,p)}(E_{2},\ldots
,E_{n+1};F)$ by the induction hypotheses and hence there is a constant $C_{1}$
such that
\[
\pi_{q;r,\ldots,r,p}(B)\leq C_{1}\Vert B\Vert\mathrm{~for~every~}%
B\in\mathcal{L}(E_{2},\ldots,E_{n+1};F).
\]
Since $E_{1}\in B(p,q,r,F)$, there is a constant $C_{2}$ such that
\[
\pi_{q;r}(v)\leq C_{2}\Vert v\Vert\mathrm{~for~every~}v\in\mathcal{L}%
(E_{1};\ell_{q}(F)).
\]
Let $A\in\mathcal{L}(E_{1},\ldots,E_{n+1};F)$. Given sequences $(x_{j}%
^{(1)})_{j=1}^{\infty}\in\ell_{r}^{w}(E_{1}),\ldots,$ $(x_{j}^{(n)}%
)_{j=1}^{\infty}\in\ell_{r}^{w}(E_{n})$ and $(x_{j}^{(n+1)})_{j=1}^{\infty}%
\in\ell_{p}^{w}(E_{n+1})$, fix $m\in\mathbb{N}$ and consider the continuous
linear operator
\[
A_{1}^{(m)}\colon E_{1}\longrightarrow\ell_{q}(F)~:~A_{1}^{(m)}(x)=\left(
(A(x,x_{j_{2}}^{(2)},\ldots,x_{j_{n+1}}^{(n+1)}))_{j_{2},\ldots,j_{n+1}=1}%
^{m},0,0,\ldots\right)  .
\]
So, $A_{1}^{(m)}$ is $(q;r)$-summing and $\pi_{q;r}(A_{1}^{(m)})\leq
C_{2}\Vert A_{1}^{(m)}\Vert$. For each $x\in B_{E_{1}}$, consider the
continuous $n$-linear mapping
\[
A_{x}^{n}\colon E_{2}\times\cdots\times E_{n+1}\longrightarrow F~:~A_{x}%
^{n}(x_{2},\ldots,x_{n+1})=A(x,x_{2},\ldots,x_{n+1}).
\]
So,
\[
\pi_{q;r,\ldots,r,p}(A_{x}^{n})\leq C_{1}\Vert A_{x}^{n}\Vert\leq C_{1}\Vert
A\Vert\Vert x\Vert\leq C_{1}\Vert A\Vert
\]
and we conclude that
\[
\left(  \sum_{j_{1}=1}^{m}\cdots\sum_{j_{n+1}=1}^{m}\left\Vert A(x_{j_{1}%
}^{(1)},\ldots x_{j_{n+1}}^{(n+1)})\right\Vert ^{q}\right)  ^{\frac{1}{q}}\leq
C_{1}C_{2}\Vert A\Vert\left(  \prod_{k=1}^{n}\left\Vert (x_{j}^{(k)}%
)_{j=1}^{m}\right\Vert _{w,r}\right)  \left\Vert (x_{j}^{(n+1)})_{j=1}%
^{m}\right\Vert _{w,p}.
\]

\end{proof}

\begin{corollary}
[Souza, 2003, P\'{e}rez-Garc\'{\i}a, 2003 ](\textrm{\cite{bombal, pgtese,
souza})} If $F$ has cotype $q$ and $E_{1},\ldots,E_{n}$ are arbitrary Banach
spaces, then
\[
\mathcal{L}(E_{1},\ldots,E_{n};F)=\mathcal{L}_{m}{}_{(q;1)}(E_{1},\ldots
,E_{n};F)\mathit{~and~}%
\]%
\[
\pi_{q;1}(A)\leq C_{q}(F)^{n}\Vert A\Vert\mathit{~for~every~}A\in
\mathcal{L}(E_{1},\ldots,E_{n};F),
\]
where $C_{q}(F)$ is the cotype $q$ constant of $F$.
\end{corollary}

\begin{proof}
Both $F$ and $\ell_{q}(F)$ have cotype $q$ (see \cite[Theorem 11.12]{djt}), so
$\mathcal{L}(E;F)=\Pi_{q;1}(E;F)$ and $\mathcal{L}(E;\ell_{q}(F))=\Pi
_{q;1}(E;\ell_{q}(F))$ for every Banach space $E$ by \cite[Corollary
11.17]{djt}.
\end{proof}

\begin{corollary}
[P\'{e}rez-Garc\'{\i}a, 2003](\textrm{\cite{bombal, pgtese})} If $E_{1}%
,\ldots,E_{n}$ are $\mathcal{L}_{1}$-spaces and $H$ is a Hilbert space, then
\[
\mathcal{L}(E_{1},\ldots,E_{n};H)=\mathcal{L}_{m,2}(E_{1},\ldots
,E_{n};H)\mathit{~and~}%
\]%
\[
\pi_{2}(A)\leq K_{G}^{n}\Vert A\Vert\mathit{~for~every~}A\in\mathcal{L}%
(E_{1},\ldots,E_{n};H),
\]
where $K_{G}$ stands for the Grothendieck constant.
\end{corollary}

\begin{proof}
From \cite[Ex. 23.17(a)]{DF} we know that $H$ and $\ell_{2}(H)$ are
$\mathcal{L}_{2}$-spaces, so $\mathcal{L}(E;H)=\Pi_{2;2}(E;H)$ and
$\mathcal{L}(E;\ell_{2}(H))=\Pi_{2;2}(E;\ell_{2}(H))$ for every $\mathcal{L}%
_{1}$-space $E$ by \cite[Theorems 3.1 and 2.8]{djt}.
\end{proof}

\begin{corollary}
[P\'{e}rez-Garc\'{\i}a, 2003](\textrm{\cite{bombal, pgtese})} If $F$ has
cotype $2$ and $E_{1},\ldots,E_{n}$ are $\mathcal{L}_{\infty}$-spaces, then
$\mathcal{L}(E_{1},\ldots,E_{n};F)=\mathcal{L}_{m,2}(E_{1},\ldots,E_{n};F).$
\end{corollary}

\begin{proof}
From \cite[Theorem 11.12]{djt} we know that $F$ and $\ell_{2}(F)$ have cotype
$2$, so $\mathcal{L}(E;F)=\Pi_{2;2}(E;F)$ and $\mathcal{L}(E;\ell_{2}%
(F))=\Pi_{2;2}(E;\ell_{2}(F))$ for every $\mathcal{L}_{\infty}$-space $E$ by
\cite[Theorem 11.14(a)]{djt}.
\end{proof}

By invoking \cite[Theorem 11.14(b)]{djt} instead of \cite[Theorem
11.14(a)]{djt} we get:

\begin{corollary}
If $F$ has cotype $q>2$, $E_{1},\ldots,E_{n}$ are $\mathcal{L}_{\infty}%
$-spaces and $r<q$, then $\mathcal{L}(E_{1},\ldots,E_{n};F)=\mathcal{L}%
_{m(q,r)}(E_{1},\ldots,E_{n};F).$
\end{corollary}

Very recently, in a remarkable paper \cite{Def2}, A. Defant, D. Popa and U.
Schwarting introduced the notion of coordinatewise multiple summing operators
and, among various interesting results, presented the following vector-valued
generalization of Bohnenblust-Hille inequality. Below, a multilinear map $U$
$\in\mathcal{L}(^{n}E_{1},...,E_{n};F)$ is separately $(r,1)$-summing if it is
absolutely $(r,1)$-summing in each coordinate separately.

\begin{theorem}
[Defant, Popa, Schwarting, 2010]Let $F$ be a Banach space with cotype $q$, and
$1\leq r<q$. Then each separately $(r,1)$-summing $U$ $\in\mathcal{L}%
(E_{1},...,E_{n};F)$ is multiple $(\frac{qrn}{q+(n-1)r},1)$-summing.
\end{theorem}

Using $F=\mathbb{K}$, $q=2$ and $r=1$ in the above theorem, Bohnenblust-Hille
Theorem is recovered.

A last comment about the richness of applications of the class of absolutely
summing multilinear operators is related to tensor norms.

A. Defant and D. P\'{e}rez-Garc\'{\i}a \cite{andreas david} constructed an
$n$-tensor norm, in the sense of Grothendieck (associated to the class of
multiple $1$-summing multilinear forms) possessing the surprising property
that the $\alpha$-tensor product $\alpha(Y_{1},...,Y_{n})$ has local
unconditional structure for each choice of $n$ arbitrary $\mathcal{L}_{p_{j}}%
$-spaces $Y_{j}.$ This construction answers a question posed by J. Diestel. It
is interesting to mention that in \cite{32} it is shown that none of
Grothendieck's $14$ norms satisfies such condition.

\section{Other attempts of multi-ideals related to absolutely summing
operators: an overview}

In the last decade several classes of multilinear maps have been investigated
as extensions of the linear concept of absolutely summing operators (for works
comparing these different classes we refer to \cite{CD, davidarchiv}).
Depending on the properties that a given class possesses, this class is
usually compared with the original linear ideal and, in some sense, qualified
as a good (or bad) extension of the linear ideal. In this direction, the
ideals of dominated multilinear operators and multiple summing multilinear
operators are mostly classified as nice generalizations of absolutely summing
linear operators. Of course, the evaluation of what properties are important
or not has a subjective component, but some classical properties of absolutely
summing operators are naturally expected to hold in the context of a
reasonable multilinear generalization.

The usual procedure in the multilinear and polynomial theory of absolutely
summing operators is to define a class and study their properties. The final
sections of this paper have a different purpose; we elect some properties that
we consider fundamental and investigate which classes satisfy them (specially
if there exist maximal and minimal classes, in a sense that will be clear soon).

Below we sketch an overview of the different multilinear approaches to
summability of operators which have arisen in the last years:

\subsection{Dominated multilinear operators: the first attempt}

If $p\geq1,$ $T\in\mathcal{L}(E_{1},...,E_{n};F)$ is said to be $p$-dominated
$(T\in\mathcal{L}_{d,p}(E_{1},...,E_{n};F))$ if $\left(  T(x_{j}^{1}%
,...,x_{j}^{n})\right)  _{j=1}^{\infty}\in\ell_{p/n}(F)$ whenever $(x_{j}%
^{k})_{j=1}^{\infty}\in\ell_{p}^{w}(E_{k}).$ This concept was essentially
introduced by Pietsch and explored in \cite{AlencarMatos, anais, sch} and has
strong similarity with the original linear ideal of absolutely summing
operators; during some time (before the emergence of the class if multiple
summing multilinear operators) this ideal seemed to be considered as the most
promising multilinear approach to summability (however, as it will be clear
soon, this class is, is some sense, too small). The terminology
\textquotedblleft$p$-dominated\textquotedblright\ is justified by the
Pietsch-Domination type theorem (a detailed proof can be found in
\cite{tesina} or as a consequence of a more general result \cite{pss}):

\begin{theorem}
[Pietsch, Geiss, 1985](\cite{geisse}) $T\in\mathcal{L}(E_{1},...,E_{n};F)$ is
$p$-dominated if and only if there exist $C\geq0$ and regular probability
measures $\mu_{j}$ on the Borel $\sigma$-algebras of $B_{E_{j}^{^{\prime}}}$
endowed with the weak star topologies such that
\[
\left\Vert T\left(  x_{1},...,x_{n}\right)  \right\Vert \leq C\prod
\limits_{j=1}^{n}\left(  \int_{B_{E_{j}^{\prime}}}\left\vert \varphi\left(
x_{j}\right)  \right\vert ^{p}d\mu_{j}\left(  \varphi\right)  \right)  ^{1/p}%
\]
for every $x_{j}\in E_{j}$ and $j=1,...,n$.
\end{theorem}

\begin{corollary}
If $1\leq p\leq q<\infty$, then $\mathcal{L}_{d,p}\subset\mathcal{L}_{d,q}.$
\end{corollary}

This class has several other similarities with the linear concept of
absolutely summing operators. We mention two results whose proofs mimic the
linear analogues:

\begin{theorem}
[Mel\'{e}ndez-Tonge, 1999](\cite{MT}) Let $2<p<r^{\ast}<\infty.$ Let $n$ be a
positive integer and $F$ be a Banach space. Then%
\[
\mathcal{L}_{d,1}(^{n}\ell_{p};F)=\mathcal{L}_{d,r}(^{n}\ell_{p};F).
\]

\end{theorem}

\begin{theorem}
[Extrapolation Theorem, 2005](\cite{irishd}) If $1<r<p<\infty$ and $E$ is a
Banach space such that%
\[
\mathcal{L}_{d,p}(^{n}E;\ell_{p})=\mathcal{L}_{d,r}(^{n}E;\ell_{p}),
\]
then
\[
\mathcal{L}_{d,p}(^{n}E;F)=\mathcal{L}_{d,1}(^{n}E;F)
\]
for every Banach space $F.$\bigskip
\end{theorem}

A consequence of Grothendieck's Inequality ensures a rare coincidence
situation for this class (this result seems to be part of the folklore of the theory):

\begin{theorem}
\label{yb}$\mathcal{L}_{d,2}(^{2}c_{0};\mathbb{K})=\mathcal{L}(^{2}%
c_{0};\mathbb{K}).$
\end{theorem}

\begin{proof}
(Real case) It suffices to deal with $A:\ell_{\infty}^{m}\times\ell_{\infty
}^{m}\rightarrow\mathbb{R}$ with $\left\Vert A\right\Vert \leq1$. Note that%
\[
\left\vert A(x,y)\right\vert =\left\vert A\left(
%TCIMACRO{\dsum \limits_{i=1}^{m}}%
%BeginExpansion
{\displaystyle\sum\limits_{i=1}^{m}}
%EndExpansion
x_{i}e_{i},%
%TCIMACRO{\dsum \limits_{i=1}^{m}}%
%BeginExpansion
{\displaystyle\sum\limits_{i=1}^{m}}
%EndExpansion
y_{i}e_{i}\right)  \right\vert =\left\vert
%TCIMACRO{\dsum \limits_{i,j=1}^{m}}%
%BeginExpansion
{\displaystyle\sum\limits_{i,j=1}^{m}}
%EndExpansion
A(e_{i},e_{j})x_{i}y_{j}\right\vert .
\]
Let $a_{ij}=A(e_{i},e_{j})$ and $(x_{k})_{k=1}^{N},(y_{k})_{k=1}^{N}\in
\ell_{2}^{w}(\ell_{\infty}^{m})$ be so that $\left\Vert (x_{k})_{k=1}%
^{N}\right\Vert _{w,2},\left\Vert (y_{k})_{k=1}^{N}\right\Vert _{w,2}\leq1$,
with%
\[
x_{k}=(x_{k}^{(1)},...,x_{k}^{(m)})\text{ and }y_{k}=(y_{k}^{(1)}%
,...,y_{k}^{(m)}).
\]
Hence, for $i,j=1,...,m,$ consider%
\[
\widetilde{x_{i}}:=(x_{1}^{(i)},...,x_{N}^{(i)})\in\ell_{2}^{N}\text{ and
}\widetilde{y_{j}}:=(y_{1}^{(j)},...,y_{N}^{(j)})\in\ell_{2}^{N}.
\]
It is well-known (see, for example, \cite[Proposici\'{o}n 5.18]{tesina}) that%
\[
\left\Vert (x_{k})_{k=1}^{N}\right\Vert _{w,2}^{2}=\max_{1\leq i\leq
m}\left\Vert \widetilde{x_{i}}\right\Vert ^{2}\text{ and }\left\Vert
(y_{k})_{k=1}^{N}\right\Vert _{w,2}^{2}=\max_{1\leq j\leq m}\left\Vert
\widetilde{y_{j}}\right\Vert ^{2}.
\]
So we have $\left\Vert \widetilde{x_{i}}\right\Vert \leq1,\left\Vert
\widetilde{y_{j}}\right\Vert \leq1$ for every $i,j=1,...,m,$ and, since
$\left\Vert A\right\Vert \leq1,$ from Grothendieck's Inequality we have%
\[
\left\vert
%TCIMACRO{\dsum \limits_{i,j=1}^{m}}%
%BeginExpansion
{\displaystyle\sum\limits_{i,j=1}^{m}}
%EndExpansion
a_{ij}<\widetilde{x_{i}},\widetilde{y_{j}}>\right\vert \leq K_{G},
\]
and therefore%
\[
\left\vert
%TCIMACRO{\dsum \limits_{i,j=1}^{m}}%
%BeginExpansion
{\displaystyle\sum\limits_{i,j=1}^{m}}
%EndExpansion
a_{ij}%
%TCIMACRO{\dsum \limits_{k=1}^{N}}%
%BeginExpansion
{\displaystyle\sum\limits_{k=1}^{N}}
%EndExpansion
x_{k}^{(i)}y_{k}^{(j)}\right\vert \leq K_{G}%
\]
i.e.,%
\[
\left\vert
%TCIMACRO{\dsum \limits_{k=1}^{N}}%
%BeginExpansion
{\displaystyle\sum\limits_{k=1}^{N}}
%EndExpansion
\left(
%TCIMACRO{\dsum \limits_{i,j=1}^{m}}%
%BeginExpansion
{\displaystyle\sum\limits_{i,j=1}^{m}}
%EndExpansion
a_{ij}x_{k}^{(i)}y_{k}^{(j)}\right)  \right\vert \leq K_{G}%
\]
and%
\[
\left\vert
%TCIMACRO{\dsum \limits_{k=1}^{N}}%
%BeginExpansion
{\displaystyle\sum\limits_{k=1}^{N}}
%EndExpansion
A\left(  x_{k},y_{k}\right)  \right\vert =\left\vert
%TCIMACRO{\dsum \limits_{k=1}^{N}}%
%BeginExpansion
{\displaystyle\sum\limits_{k=1}^{N}}
%EndExpansion
A\left(
%TCIMACRO{\dsum \limits_{i=1}^{m}}%
%BeginExpansion
{\displaystyle\sum\limits_{i=1}^{m}}
%EndExpansion
x_{k}^{(i)}e_{i},%
%TCIMACRO{\dsum \limits_{j=1}^{m}}%
%BeginExpansion
{\displaystyle\sum\limits_{j=1}^{m}}
%EndExpansion
y_{k}^{(j)}e_{j}\right)  \right\vert \leq K_{G}.
\]
Since $x_{k}$ can be replaced by $\varepsilon_{k}x_{k}$ with $\varepsilon
_{k}=1$ or $-1$, we can conclude that
\[%
%TCIMACRO{\dsum \limits_{k=1}^{N}}%
%BeginExpansion
{\displaystyle\sum\limits_{k=1}^{N}}
%EndExpansion
\left\vert A\left(  x_{k},y_{k}\right)  \right\vert \leq K_{G}.
\]
\bigskip
\end{proof}

In fact the result above is valid for $\mathcal{L}_{\infty}$ spaces instead of
$c_{0}$. For a direct proof of this result to $C(K)$ spaces we refer to
\cite{irish}.

It is also known that dominated multilinear maps satisfy a Dvoretzky-Rogers
type theorem ($\mathcal{L}_{d,p}(^{n}E;E)=\mathcal{L}(^{n}E;E)$ if and only if
$\dim E<\infty$). Recent results show that this class is too small, in some
sense (coincidence situations are almost impossible). The proof of the next
result presented here is different from the original \cite{Jar}, and appears
in \cite{belg}:

\begin{theorem}
[Jarchow, Palazuelos, P\'{e}rez-Garc\'{\i}a and Villanueva, 2007]%
\label{ffy}(\cite{Jar}) For every $n\geq3$ and every $p\geq1$ and every
infinite dimensional Banach space $E$ there exists $T\in\mathcal{L}%
(^{n}E;\mathbb{K})$ that fails to be $p$-dominated.
\end{theorem}

\begin{proof}
Suppose that every $T\in\mathcal{L}(^{3}E;\mathbb{K})$ is $p$-dominated. From
\cite[Lemma 3.4]{irish} one can conclude that every continuous linear operator
from $E$ to $\mathcal{L}(E;\mathcal{L}(^{2}E;\mathbb{K}))$ is $p$-summing.
From \cite[Proposition 19.17]{djt} we know that $\mathcal{L}(^{2}%
E;\mathbb{K})$ has no finite cotype, but from Theorem \ref{oik} (iii) this is
not possible. Since the result is true for $n=3$, it is easy to conclude that
it is true for $n>3$.
\end{proof}

For polynomial versions of this result we refer to \cite{PAMS, PAMS2} and for
more results on dominated multilinear operators/polynomials we refer to
\cite{irish, PAMS, cg, Jar, MT} and references therein.

Since Theorem \ref{ffy} is valid for $n\geq3$, a natural question is: are
there coincidence situations for $n=2$ different from the obvious variations
of Theorem \ref{yb}? The answer is yes:

\begin{theorem}
[Botelho, Pellegrino, Rueda, 2010](\cite{jap})\label{jja} Let $E$ be a cotype
$2$ space. Then $E\widehat{\otimes}_{\pi}E=E\widehat{\otimes}_{\varepsilon}E$
if and only if $\mathcal{L}_{d,1}(^{2}E;\mathbb{K})=\mathcal{L}(^{2}%
E;\mathbb{K})$.
\end{theorem}

The existence of spaces fulfilling the hypotheses of Theorem \ref{jja} is
assured by G. Pisier \cite{18}. Also, $\cot E=2$ is a necessary condition for
Theorem \ref{jja} since in \cite{jap} it is also proved that%
\[
\mathcal{L}_{d,1}(^{2}E;\mathbb{K})=\mathcal{L}(^{2}E;\mathbb{K}%
)\Rightarrow\cot E=2.
\]

\subsection{Semi-integral multilinear operators}

If $p\geq1,$ $T\in\mathcal{L}(E_{1},...E_{n};F)$ is $p$-semi-integral
$(T\in\mathcal{L}_{si,p}(E_{1},...,E_{n};F))$ if there exists a $C\geq0$ such
that%
\[
\left(  \sum\limits_{j=1}^{m}\parallel T(x_{j}^{(1)},...,x_{j}^{(n)}%
)\parallel^{p}\right)  ^{1/p}\leq C\left(  \sup_{\left(  \varphi
_{1},..,\varphi_{_{n}}\right)  \in B_{E_{1}^{\ast}}\times\cdots\times
B_{E_{n}^{\ast}}}\sum\limits_{j=1}^{m}\mid\varphi_{1}(x_{j}^{(1)}%
)...\varphi_{n}(x_{j}^{(n)})\mid^{p}\right)  ^{1/p}%
\]
for every $m\in\mathbb{N}$, $x_{j}^{(l)}\in E_{l}$ with $l=1,...,n$ and
$j=1,...,m.$

This ideal goes back to the research report \cite{AlencarMatos} of R. Alencar
and M.C. Matos and was explored in \cite{CD}. As in the case of $p$-dominated
multilinear operators, a Pietsch Domination theorem is valid in this context
(for a proof we mention \cite{CD}, although the result is inspired by the case
$p=1$ from Alencar-Matos paper \cite{AlencarMatos}; see also \cite{BPRn} for a
recent general argument):

\begin{theorem}
[Alencar, Matos, 1989 and \c{C}aliskan, Pellegrino, 2007]$T\in\mathcal{L}%
(E_{1},...E_{n};F)$ is $p$-semi-integral if and only if there exist $C\geq0$
and a regular probability measure $\mu$ on the Borel $\sigma-$algebra
$\mathcal{B}(B_{E_{1}^{^{\ast}}}\times\cdots\times$ $B_{E_{n}^{^{\ast}}})$ of
$B_{E_{1}^{^{\ast}}}\times\cdots\times$ $B_{E_{n}^{^{\ast}}}$ endowed with the
product of the weak star topologies $\sigma(E_{l}^{\ast},E_{l}),$ $l=1,...,n,$
such that
\[
\parallel T(x_{1},...,x_{n})\parallel\leq C\left(  \int_{B_{E_{1}^{\ast}%
}\times\cdots\times B_{E_{n}^{\ast}}}\mid\varphi_{1}(x_{1})...\varphi
_{n}(x_{n})\mid^{p}d\mu(\varphi_{1},...,\varphi_{n})\right)  ^{1/p}%
\]

\end{theorem}

\begin{corollary}
If $1\leq p\leq q<\infty,$ then $\mathcal{L}_{si,p}\subset\mathcal{L}_{si,q}.$
\end{corollary}

It is well-known that, as it happens with the ideal of $p$-dominated
multilinear operators, this ideal satisfies a Dvoretzky-Rogers type theorem.

This \textquotedblleft size\textquotedblright\ of this class is strongly
connected to the \textquotedblleft size\textquotedblright\ of the class of
$p$-dominated multilinear operators. For example, in \cite{CD} it is shown
that
\begin{equation}
\mathcal{L}_{si,p}(E_{1},...,E_{n};F)\subset\mathcal{L}_{d,np}(E_{1}%
,...,E_{n};F). \label{ww}%
\end{equation}

In fact, if $T\in\mathcal{L}_{si,p}(E_{1},...,E_{n};F)$ then
\begin{align*}
\left(  \sum\limits_{j=1}^{\infty}\parallel T(x_{1}^{(j)},...,x_{n}%
^{(j)})\parallel^{p}\right)  ^{1/p}  &  \leq C\left(  \sup_{\varphi_{l}\in
B_{E_{l}^{\ast}},l=1,...,n}\sum\limits_{j=1}^{\infty}\mid\varphi_{1}%
(x_{1}^{(j)})...\varphi_{n}(x_{n}^{(j)})\mid^{p}\right)  ^{1/p}\\
&  \leq C\sup_{\varphi_{l}\in B_{E_{l}^{\ast}},l=1,...,n}\left(
\sum\limits_{j=1}^{\infty}\mid\varphi_{1}(x_{1}^{(j)})\mid^{np}\right)
^{\frac{1}{np}}...\left(  \sum\limits_{j=1}^{\infty}\mid\varphi_{n}%
(x_{n}^{(j)})\mid^{np}\right)  ^{\frac{1}{np}}\\
&  =C\left\Vert (x_{1}^{(j)})_{j=1}^{\infty}\right\Vert _{w,np}...\left\Vert
(x_{n}^{(j)})_{j=1}^{\infty}\right\Vert _{w,np}.
\end{align*}
\bigskip

In view of the \textquotedblleft small size\textquotedblright\ of the class of
$p$-dominated multilinear operators, the inclusion (\ref{ww}) might be viewed
as a bad property.

\subsection{Strongly summing multilinear operators}

If $p\geq1,$ $T\in\mathcal{L}(E_{1},...,E_{n};F)$ is strongly $p$-summing
($T\in\mathcal{L}_{ss,p}(E_{1},...,E_{n};F)$) if there exists a constant
$C\geq0$ such that
\begin{equation}
\left(  \sum\limits_{j=1}^{m}\parallel T(x_{j}^{(1)},...,x_{j}^{(n)}%
)\parallel^{p}\right)  ^{1/p}\leq C\left(  \underset{\phi\in B_{\mathcal{L}%
(E_{1},...,E_{n};\mathbb{K})}}{\sup}\sum\limits_{j=1}^{m}\mid\phi(x_{j}%
^{(1)},...,x_{j}^{(n)})\mid^{p}\right)  ^{1/p}. \label{in}%
\end{equation}
for every $m\in\mathbb{N}$, $x_{j}^{(l)}\in E_{l}$ with $l=1,...,n$ and
$j=1,...,m.$

The multi-ideal of strongly $p$-summing multilinear operators is due to V.
Dimant \cite{dimant} is perhaps the class that best translates to the
multilinear setting the properties of the original linear concept. For
example, a Grothendieck type theorem and a Pietsch-Domination type theorem are valid:

\begin{theorem}
[Dimant, 2003](\cite{dimant}) Every $T\in\mathcal{L}(^{n}\ell_{1};\ell_{2})$
is strongly $1$-summing.
\end{theorem}

\begin{theorem}
[Dimant, 2003](\cite{dimant}) $T\in\mathcal{L}\left(  E_{1},...,E_{n}%
;F\right)  $ is strongly $p$-summing if, and only if, there are a probability
measure $\mu$ on $B_{(E_{1}\otimes_{\pi}\cdots\otimes_{\pi}E_{n})^{\ast}}$,
with the weak-star topology, and a constant $C\geq0$ so that
\begin{equation}
\left\Vert T\left(  x_{1},...,x_{n}\right)  \right\Vert \leq C\left(
\int_{B_{(E_{1}\otimes_{\pi}\cdots\otimes_{\pi}E_{n})^{\ast}}}\left\vert
\varphi\left(  x_{1}\otimes\cdots\otimes x_{n}\right)  \right\vert ^{p}%
d\mu\left(  \varphi\right)  \right)  ^{\frac{1}{p}} \label{7out08c}%
\end{equation}
for all $(x_{1},...,x_{n})\in E_{1}\times\cdots\times E_{n},$
\end{theorem}

The following intriguing result shows that in special situations the class of
strongly $p$-summing multilinear maps contains the ideal of multiple
$p$-summing operators:

\begin{theorem}
[Mezrag, Saadi, 2009](\cite{mss}) Let $1<p<\infty.$ If $E_{j}$ is an
$\mathcal{L}_{p}$-space for all $j=1,...,n$ and $F$ is an $\mathcal{L}%
_{p^{\ast}}$-space, then%
\[
\mathcal{L}_{m,p^{\ast}}(E_{1},...,E_{n};F)\subset\mathcal{L}_{ss,p^{\ast}%
}(E_{1},...,E_{n};F).
\]
\bigskip
\end{theorem}

It is not hard to prove that a Dvoretzky-Rogers Theorem is also valid for this
class. Besides, the class has a nice size in the sense that no coincidence
theorem can hold for $n$-linear maps if there is no analogue for linear
operators. This indicates that this class is not \textquotedblleft
unnecessarily big\textquotedblright.

\subsection{Absolutely\emph{ }summing multilinear operators}

If $\frac{1}{p}\leq\frac{1}{q_{1}}+\cdots+\frac{1}{q_{n}},$ $T\in
\mathcal{L}(E_{1},...,E_{n};F)$ is absolutely\emph{ }$(p;q_{1},...,q_{n}%
)$-summing at the point $a=(a_{1},...,a_{n})\in E_{1}\times\cdots\times E_{n}$
when
\[
\left(  T(a_{1}+x_{j}^{(1)},...,a_{n}+x_{j}^{(n)})-T(a_{1},...,a_{n})\right)
_{j=1}^{\infty}\in\ell_{p}(F)
\]
for every $\left(  x_{j}^{(k)}\right)  _{j=1}^{\infty}\in\ell_{q_{k}}%
^{w}(E_{k}).$ This class is denoted by $\mathcal{L}_{as,(p;q_{1},...,q_{n}%
)}^{(a)}.$ When $a$ is the origin call simply absolutely $(p;q_{1},...,q_{n}%
)$-summing and represent by $\mathcal{L}_{as,(p;q_{1},...,q_{n})}$ (when
$q_{1}=\cdots=q_{n}=q$ we write $\mathcal{L}_{as,(p;q)}$ and when
$q_{1}=...=q_{n}=q=p$ we just write $\mathcal{L}_{as,p}$)$.$ In the case that
$T$ is absolutely $(p;q_{1},...,q_{n})$-summing at every $(a_{1},...,a_{n})\in
E_{1}\times\cdots\times E_{n}$ we say that $T$ is absolutely $p$-summing
everywhere and we write $T\in\mathcal{L}_{as,(p;q_{1},...,q_{n})}^{ev}%
(E_{1},...,E_{n};F)$ (when $q_{1}=\cdots=q_{n}=q$ we write $\mathcal{L}%
_{as,(p;q)}^{ev}$ and when $q_{1}=...=q_{n}=q=p$ we just write $\mathcal{L}%
_{as,p}^{ev}$)$.$

The class of absolutely $(p;q_{1},...,q_{n})$-summing operators (when $a=0$)
seems to have appeared for the first time in \cite{AlencarMatos}. The starting
point of the theory of absolutely\emph{ }summing is perhaps the result due to
A. Defant and J. Voigt (see \cite{AlencarMatos}), known as Defant-Voigt
Theorem, which asserts that every continuous multilinear form is
$(1;1,...,1)$-summing. We prove here a slightly more general version which can
be found in (\cite{port}):

\begin{theorem}
[The generalized Defant-Voigt Theorem, 2007]\label{t1}Let $A\in\mathcal{L}%
(E_{1},...,E_{n};F)$ and suppose that there exist $1\leq r<n$ and $C>0$ so
that for any $x_{1}\in E_{1},....,x_{r}\in E_{r},$ the $s$-linear ($s=n-r$)
mapping $A_{x_{1}....x_{r}}(x_{r+1},...,x_{n})=A(x_{1},...,x_{n})$ is
absolutely $(p;q_{1},...,q_{s})$-summing and
\[
\left\Vert A_{x_{1}....x_{r}}\right\Vert _{as(p;q_{1},...,q_{s})}\leq
C\left\Vert A\right\Vert \left\Vert x_{1}\right\Vert ...\left\Vert
x_{r}\right\Vert .
\]
Then $A$ is absolutely $(p;1,...,1,q_{1},...,q_{s})$-summing. In particular%
\[
\mathcal{L}(E_{1},...,E_{n};\mathbb{K})=\mathcal{L}_{as,1}(E_{1}%
,...,E_{n};\mathbb{K})
\]

\end{theorem}

\begin{proof}
(Sketch) Given $m\in\mathbb{N}$ and $x_{1}^{(1)},...,x_{1}^{(m)}\in
E_{1},....,x_{n}^{(1)},...,x_{n}^{(m)}\in E_{n}$, let us consider $\varphi
_{j}\in B_{F^{\prime}}$ such that $\left\Vert A(x_{1}^{(j)},...,x_{n}%
^{(j)})\right\Vert =\varphi_{j}(A(x_{1}^{(j)},...,x_{n}^{(j)}))$ for every
$j=1,...,m$. Fix $b_{1},...,b_{m}\in\mathbb{K}$ so that $\sum\limits_{j=1}%
^{m}\left\vert b_{j}\right\vert ^{q}=1,$ where $\frac{1}{p}+\frac{1}{q}=1,$
and
\[
\left(  \sum\limits_{j=1}^{m}\left\Vert A(x_{1}^{(j)},...,x_{n}^{(j)}%
)\right\Vert ^{p}\right)  ^{\frac{1}{p}}=\left\Vert \left(  \left\Vert
A(x_{1}^{(j)},...,x_{n}^{(j)})\right\Vert \right)  _{j=1}^{m}\right\Vert
_{p}=\sum\limits_{j=1}^{m}b_{j}\left\Vert A(x_{1}^{(j)},...,x_{n}%
^{(j)})\right\Vert .
\]
If $\lambda$ is the Lebesgue measure on $I=[0,1]^{r},$ we have%
\[
\int\nolimits_{I}\sum\limits_{j=1}^{m}\left(  \prod_{l=1}^{r}r_{j}%
(t_{l})\right)  b_{j}\varphi_{j}A(\sum\limits_{j_{1}=1}^{m}r_{j_{1}}%
(t_{1})x_{1}^{(j_{1})},...,\sum\limits_{j_{r}=1}^{m}r_{j_{r}}(t_{r}%
)x_{r}^{(j_{r})},x_{r+1}^{(j)},...,x_{n}^{(j)})d\lambda
\]%
\[
=\sum\limits_{j,j_{1},...,j_{r}=1}^{m}b_{j}\varphi_{j}A(x_{1}^{(j_{1}%
)},...,x_{r}^{(j_{r})},x_{r+1}^{(j)},...,x_{n}^{(j)})\int\limits_{0}^{1}%
r_{j}(t_{1})r_{j_{1}}(t_{1})dt_{1}...\int\limits_{0}^{1}r_{j}(t_{r})r_{j_{r}%
}(t_{r})dt_{r}\hspace*{1.3em}%
\]%
\[
=\sum\limits_{j=1}^{m}\sum\limits_{j_{1}=1}^{m}...\sum\limits_{j_{r}=1}%
^{m}b_{j}\varphi_{j}A(x_{1}^{(j_{1})},...,x_{r}^{(j_{r})},x_{r+1}%
^{(j)},...,x_{n}^{(j)})\delta_{jj_{1}}...\delta_{jj_{r}}=\sum\limits_{j=1}%
^{m}b_{j}\varphi_{j}A(x_{1}^{(j)},...,x_{n}^{(j)}).
\]
For $z_{l}=\sum\limits_{j=1}^{m}r_{j}(t_{l})x_{l}^{(j)}$, $l=1,...,r,$ we get
\begin{align*}
&  \left(  \sum\limits_{j=1}^{m}\left\Vert A(x_{1}^{(j)},...,x_{n}%
^{(j)})\right\Vert ^{p}\right)  ^{\frac{1}{p}}=\sum\limits_{j=1}^{m}%
b_{j}\varphi_{j}A(x_{1}^{(j)},...,x_{n}^{(j)})\\
&  \leq\int\nolimits_{I}\left\vert \sum\limits_{j=1}^{m}\left(  \prod
_{l=1}^{r}r_{j}(t_{l})\right)  b_{j}\varphi_{j}A(\sum\limits_{j_{1}=1}%
^{m}r_{j_{1}}(t_{1})x_{1}^{(j_{1})},...,\sum\limits_{j_{r}=1}^{m}r_{j_{r}%
}(t_{r})x_{r}^{(j_{r})},x_{r+1}^{(j)},...,x_{n}^{(j)})\right\vert d\lambda
\end{align*}
and after standard calculations we get%
\[
\left(  \sum\limits_{j=1}^{m}\left\Vert A(x_{1}^{(j)},...,x_{n}^{(j)}%
)\right\Vert ^{p}\right)  ^{\frac{1}{p}}\leq C\left\Vert A\right\Vert \left(
\prod_{l=1}^{r}\left\Vert (x_{l}^{(j)})_{j=1}^{m}\right\Vert _{w,1}\right)
\left(  \prod_{l=r+1}^{n}\left\Vert (x_{l}^{(j)})_{j=1}^{m}\right\Vert
_{w,q_{l}}\right)  .\text{ }%
\]
%\end{proof}\bigskip

\end{proof}

Using a generalized version of Grothendieck's Inequality, D. P\'{e}rez
Garc\'{\i}a proved a striking generalization of Theorem \ref{yb}:

\begin{theorem}
[P\'{e}rez-Garc\'{\i}a, 2002](\cite{tesina, trace}) \label{blaw}%
$\mathcal{L}_{as,(1,2)}(^{n}c_{0};\mathbb{K})=\mathcal{L}(^{n}c_{0}%
;\mathbb{K})$ for every $n\geq2.$
\end{theorem}

A recent result from Blasco et al \cite{bla2} shows that the crucial cases of
Theorem \ref{blaw} are precisely the cases $n=2$ and $n=3:$

\begin{theorem}
[Blasco, Botelho, Pellegrino, Rueda, 2010]Let $1\leq r\leq2$. If
$\mathcal{L}(^{2}E;\mathbb{K})=\mathcal{L}_{as,(1,r)}(^{2}E;\mathbb{K})$ and
$\mathcal{L}(^{3}E;\mathbb{K})=\mathcal{L}_{as,(1,r)}(^{3}E;\mathbb{K})$, then%
\[
\mathcal{L}(^{n}E;\mathbb{K})=\mathcal{L}_{as,(1,r)}(^{n}E;\mathbb{K})
\]
for every $n\geq2.$
\end{theorem}

\begin{proof}
(Sketch of the proof when $n$ is odd) Induction: Suppose that the result is
valid for a fixed odd $k$ and we shall prove that it is also true for $k+2$.
Let $T\in\mathcal{L}(^{k+2}E;\mathbb{K})$ and consider
\begin{align*}
F  &  =E\hat{\otimes}_{\pi}\cdots\hat{\otimes}_{\pi}E\text{ (}k\text{
times)}\\
G  &  =E\hat{\otimes}_{\pi}E.
\end{align*}
Consider a bilinear form
\[
B\in\mathcal{L}(F,G;\mathbb{K})
\]
so that
\[
B(x^{1}\otimes\cdots\otimes x^{k},x^{k+1}\otimes x^{k+2})=T(x^{1}%
,...,x^{k+2}).
\]
Let $x_{j}^{(s)}\in E$ for $j=1,...,m$ and $s=1,...,k+2$. Using Defant-Voigt
Theorem for $B$ and the induction hypothesis one can found a constant $C$ so
that
\begin{align*}
\lefteqn{\sum_{j=1}^{m}\left\vert T(x_{j}^{(1)},\ldots,x_{j}^{(k+2)}%
)\right\vert }  &  & \\
&  &  &  =\sum_{j=1}^{m}\left\vert B(x_{j}^{(1)}\otimes\ldots\otimes
x_{j}^{(k)},x_{j}^{(k+1)}\otimes x_{j}^{(k+2)})\right\vert \\
&  &  &  \leq\pi_{(1;1)}(B)\left\Vert (x_{j}^{(1)}\otimes\ldots\otimes
x_{j}^{(k)})_{j=1}^{m}\right\Vert _{\ell_{1}^{w}(E\hat{\otimes}_{\pi}%
\cdots\hat{\otimes}_{\pi}E)}\left\Vert (x_{j}^{k+1}\otimes x_{j}^{k+2}%
)_{j=1}^{m}\right\Vert _{\ell_{1}^{w}(E\hat{\otimes}_{\pi}E)}\\
&  &  &  \leq C\left\Vert B\right\Vert \left\Vert (x_{j}^{(1)})_{j=1}%
^{m}\right\Vert _{\ell_{r}^{w}(E)}\cdots\left\Vert (x_{j}^{(k+2)})_{j=1}%
^{m}\right\Vert _{\ell_{r}^{w}(E_{{}})}\\
&  &  &
\end{align*}
and the proof is done$.$
\end{proof}

For general Banach spaces, the class $\mathcal{L}_{as,(1,2)}(^{n}%
E;\mathbb{K})$ also plays an important role, as an \textquotedblleft upper
bound\textquotedblright\ for the classes of $p$-dominated multilinear mappings
\cite{fmat, pgtese}:

\begin{theorem}
[Floret, Matos, 1995 (complex case), P\'{e}rez-Garc\'{\i}a, 2003]Let
$n\in\mathbb{N}$, $n\geq2$ and $p\geq1.$ If $E$ is a Banach space, then%
\[
\mathcal{L}_{d,p}(^{n}E;\mathbb{K})\subset\mathcal{L}_{as,(1,2)}%
(^{n}E;\mathbb{K}).
\]

\end{theorem}

At a first glance the concept of absolutely summing multilinear operator seems
to be the natural multilinear definition of absolute summability. However it
is easy to find bad properties which makes the ideal very different from the
linear ideal.

For example, no general Inclusion Theorem is valid. In fact, Defant-Voigt
Theorem ensures that
\[
\mathcal{L}_{as,1}(^{2}\ell_{2};\mathbb{K})=\mathcal{L}(^{n}\ell
_{2};\mathbb{K})
\]
but it is easy to show that%
\[
\mathcal{L}_{as,2}(^{2}\ell_{2};\mathbb{K})\neq\mathcal{L}(^{2}\ell
_{2};\mathbb{K}).
\]
Besides, contrary to the linear case, several coincidence theorems hold, and
this behavior removes the linear essence from this class. For example,
Grothendieck%
%TCIMACRO{\U{b4}}%
%BeginExpansion
\'{}%
%EndExpansion
s Theorem is valid but there are several other coincidence situations with
absolutely no linear analogue, as%
\begin{equation}
\mathcal{L}_{as,1}(^{n}\ell_{2};F)=\mathcal{L}(^{n}\ell_{2};F) \label{kio}%
\end{equation}
for all $n\geq2$ and all $F$. Since (\ref{kio}) is not true for $n=1$, from
now on we call coincidence situations as (\ref{kio}) by \textquotedblleft
artificial coincidence situation\textquotedblright.

Moreover, the polynomial version of this class is not an holomorphy type (this
is a bad property!) and, in the terminology of \cite{muro}, this bad property
is reinforced since this class is not compatible with the linear ideal of
absolutely summing operators. Despite its bad properties, this class has some
challenging problems (see, for example, \cite{michels, junek, danielstudia}).

As it occurs for multiple summing multilinear operators, in (\cite{michels})
it was shown that a full Inclusion Theorem is valid when the spaces from the
domain are $\mathcal{L}_{\infty}$-spaces:

\begin{theorem}
[Botelho, Michels, Pellegrino, 2010]Let $1\leq p\leq q\leq\infty$ and
$E_{1},\ldots,E_{n}$ {be} $\mathcal{L}_{\infty}$-spaces. Then
\[
\mathcal{L}_{as,p}(E_{1},\ldots,E_{n};F)\subset\mathcal{L}_{as,q}(E_{1}%
,\ldots,E_{n};F).
\]

\end{theorem}

In some cases, surprisingly, the inclusion theorem holds in the opposite
direction than the expected \cite{junek} (i.e. if $p$ increases, the ideal decreases):

\begin{theorem}
[Junek, Matos, Pellegrino, 2008]\label{tt}If $E$ has cotype $2,$ $F$ is any
Banach space and $n\geq2,$ then
\[
\mathcal{L}_{as,q}(^{n}E;F)\subset\mathcal{L}_{as,p}(^{n}E;F)
\]
holds true for $1\leq p\leq q\leq2$.
\end{theorem}

The class of everywhere absolutely $p$-summing multilinear operators was
introduced by M.C. Matos \cite{nachmatos} but he credits the idea to Richard
Aron. It is easy to show that $\mathcal{L}_{m,p}\subset\mathcal{L}_{as,p}%
^{ev}$ and, as it occurs for $\mathcal{L}_{ss,p}$ and $\mathcal{L}_{m,p},$
this class has no artificial coincidence theorem (a proof can be found in
\cite{spp}):

\begin{proposition}
If $\mathcal{L}(E_{1},\ldots,E_{n};F)=\mathcal{L}_{as,(q;p_{1},\ldots,p_{n}%
)}^{ev}(E_{1},\ldots,E_{n};F)$, then
\[
\mathcal{L}(E_{j};F)=\Pi_{q,p_{j}}(E_{j};F),j=1,\ldots,n.
\]

\end{proposition}

\subsection{Strongly multiple summing multilinear operators: the last attempt}

If $p\geq1,$ $T\in\mathcal{L}(E_{1},...,E_{n};F)$ is strongly multiple
$p$-summing ($T\in\mathcal{L}_{sm,p}(E_{1},...,E_{n};F)$) if there exists
$C\geq0$ such that
\begin{equation}
\left(  \sum\limits_{j_{1},...,j_{n}=1}^{m}\parallel T(x_{j_{1}}%
^{(1)},...,x_{j_{n}}^{(n)})\parallel^{p}\right)  ^{1/p}\leq C\left(
\underset{\phi\in B_{\mathcal{L}(E_{1},...,E_{n};\mathbb{K})}}{\sup}%
\sum\limits_{j_{1},...,j_{n}=1}^{m}\mid\phi(x_{j_{1}}^{(1)},...,x_{j_{n}%
}^{(n)})\mid^{p}\right)  ^{1/p}%
\end{equation}
for every $m\in\mathbb{N}$, $x_{j_{l}}^{(l)}\in E_{l}$ with $l=1,...,n$ and
$j_{l}=1,...,m.$\newline

The multi-ideal of strongly multiple $p$-summing multilinear operators was
introduced in \cite{port} and has not been explored since then.\ It contains the
ideals $\mathcal{L}_{ss,p}$ and $\mathcal{L}_{m,p}.$ All nice properties from
$\mathcal{L}_{ss,p}$ are also valid except perhaps for versions of Pietsch
Domination Theorem (and inclusion theorem) which are unknown. The size of this
class is potentially better than the sizes of $\mathcal{L}_{ss,p}$ and
$\mathcal{L}_{m,p}$ since despite containing these two classes, it is also
known that for this class no coincidence theorem can hold for $n$-linear maps
if there is no analogue for linear operators. So, even having a nice size,
this class has no artificial coincidence results, with no linear analogue.

\section{Desired properties for a nice multi-ideal extension of absolutely
summing operators}

In \cite{port, CD} and it is shown that
\begin{align*}
\mathcal{L}_{d,p}  &  \subset\mathcal{L}_{si,p}\subset\mathcal{L}_{m,p}%
\subset\mathcal{L}_{as,p}^{ev}\subset\mathcal{L}_{as,p}.\\
\mathcal{L}_{d,p}  &  \subset\mathcal{L}_{si,p}\subset\mathcal{L}%
_{ss,p}\subset\mathcal{L}_{sm,p}.\\
\mathcal{L}_{d,p}  &  \subset\mathcal{L}_{si,p}\subset\mathcal{L}_{m,p}%
\subset\mathcal{L}_{sm,p}.
\end{align*}

It is not difficult to show that $\mathcal{L}_{d,p}$ is cud, csm and it is
well known that the Dvoretzky-Rogers Theorem is true, and also a Pietsch
Domination Theorem (and, of course, the inclusion theorem) holds. On the other
hand, as we have mentioned before this class is small and the Grothendieck
Theorem is not true. As the above table shows the class $\mathcal{L}_{sm,p}$
is much bigger and from \cite{port} we know that this class is cud, csm, and
Dvoretzky-Rogers Theorem and Grothendieck's $\ell_{1}$-$\ell_{2}$ Theorem are
valid. More generally, this class contains the better-known class of multiple
summing multilinear operators and hence it inherits all the known coincidence
theorems for the class of multiple summing operators. In some sense, it is
natural to expect that all reasonable multilinear extensions $\mathcal{M}%
=(\mathcal{M}_{p})_{p\geq1}$ of the ideal of absolutely summing operators
should satisfy $\mathcal{L}_{d,p}\subset\mathcal{M}_{p}\subset\mathcal{L}%
_{sm,p}.$

Below we list the properties of each class:

\bigskip

$%
\begin{bmatrix}
\text{Property/ Class} & \mathcal{L}_{d,p} & \mathcal{L}_{si,p} &
\mathcal{L}_{ss,p} & \mathcal{L}_{m,p} & \mathcal{L}_{sm,p} & \mathcal{L}%
_{as,p} & \mathcal{L}_{as,p}^{ev}\\
\text{cud} & \text{Yes} & \text{Yes} & \text{\textbf{Yes}} & \text{Yes} &
\text{Yes} & \text{No} & \text{Yes}\\
\text{csm} & \text{Yes} & \text{Yes} & \text{\textbf{Yes}} & \text{Yes} &
\text{Yes} & \text{Yes} & \text{Yes}\\
\text{Grothendieck Theorem} & \text{No} & \text{No} & \text{\textbf{Yes}} &
\text{Yes} & \text{Yes} & \text{Yes} & \text{Yes}\\
\text{Inclusion Theorem} & \text{Yes} & \text{Yes} & \text{\textbf{Yes}} &
\text{No} & ? & \text{No} & \text{No}\\
\text{Dvoretzky-Rogers Theorem} & \text{Yes} & \text{Yes} & \text{\textbf{Yes}%
} & \text{Yes} & \text{Yes} & \text{No} & \text{Yes}\\
\mathcal{L}_{d,p}\subset\cdot\subset\mathcal{L}_{sm,p} & \text{Yes} &
\text{Yes} & \text{\textbf{Yes}} & \text{Yes} & \text{Yes} & \text{No} &
\text{?}%
\end{bmatrix}
$\bigskip

Taking into account the main properties of the linear ideal of absolutely
summing operators, we propose the following concept of \textquotedblleft
desired generalization of $(\Pi_{p})_{p\geq1}$\textquotedblright:

\begin{definition}
A family of normed ideals of multilinear mappings $(\mathcal{M}_{p})_{p\geq1}$
is a desired generalization of $(\Pi_{p})_{p\geq1}$ if

\begin{itemize}
\item (i) $\mathcal{L}_{d,p}\subset\mathcal{M}_{p}\subset\mathcal{L}_{sm,p}$
for all $p$ and the inclusions have norm $\leq1.$

\item (ii) $\mathcal{M}_{p}$ is csm for all $p.$

\item (iii) $\mathcal{M}_{p}$ is cud for all $p.$

\item (iv) $\mathcal{M}_{p}\subset\mathcal{M}_{q}$ whenever $p\leq q.$

\item (v) Grothendieck's Theorem and a Dvoretzky-Rogers theorem are valid.
\end{itemize}
\end{definition}

First, observe that if $\mathcal{L}_{d,p}\subset\mathcal{M}_{p}\subset
\mathcal{L}_{sm,p}$ then Dvoretzky-Rogers theorem is valid.

Note that, accordingly to the above table, $(\mathcal{L}_{ss,p})_{p\geq1}$ is
a desirable generalization of the family $(\Pi_{p})_{p\geq1}$.\bigskip\ A
desired generalization will be called \textquotedblleft desired
family\textquotedblright.

\begin{definition}
A desired family $\mathcal{M}=(\mathcal{M}_{p})_{p\geq1}$ is maximal if
whenever $\mathcal{M}_{p}\subset\mathcal{I}_{p}$ for all $p$ $($and the
inclusion has norm $\leq1)$ and $(\mathcal{I}_{p})_{p\geq1}$ is a desired
family, then $\mathcal{M}_{p}=\mathcal{I}_{p}$ for all $p.$

Similarly, a desired family $\mathcal{M}=(\mathcal{M}_{p})_{p\geq1}$ is
minimal if whenever $\mathcal{I}_{p}\subset\mathcal{M}_{p}$ for all $p$ $($and
the inclusion has norm $\leq1)$ and $(\mathcal{I}_{p})_{p\geq1}$ is a desired
family$,$ then $\mathcal{M}_{p}=\mathcal{I}_{p}$ for all $p.$
\end{definition}

\begin{theorem}
There exists a desired family of multilinear mappings which is maximal.
\end{theorem}

\begin{proof}
Let
\[
D=\left\{  \mathcal{M}^{\lambda}=(\mathcal{M}_{p}^{\lambda})_{p\geq
1}:\mathcal{M}^{\lambda}\ \text{is a desired family for every }\lambda
\in\Lambda\right\}  .
\]
In $D$ we consider the partial order
\begin{equation}
\mathcal{M}^{\lambda_{1}}\leq\mathcal{M}^{\lambda_{2}}\Leftrightarrow
\mathcal{M}_{p}^{\lambda_{1}}\subseteq\mathcal{M}_{p}^{\lambda_{2}%
}\ \text{and}\ \left\Vert \cdot\right\Vert _{\mathcal{M}_{p}^{\lambda_{2}}%
}\leq\left\Vert \cdot\right\Vert _{\mathcal{M}_{p}^{\lambda_{1}}}\text{ for
all }p\geq1. \label{ulty}%
\end{equation}

Note that $D\neq{\varnothing}$ since $(\mathcal{L}_{ss,p})_{p\geq1}\in D.$ We
just need to show that Zorn's Lemma is applicable in order to yield the
existence of a maximal family.

If $O\subset D$ is totally ordered and $\Lambda_{O}=\{\lambda\in
\Lambda:\mathcal{M}^{\lambda}=(\mathcal{M}_{p}^{\lambda})_{p\geq1}\in O\}$,
consider the class%
\[
\mathcal{U}=(\mathcal{U}_{p})_{p\geq1},
\]
where, for each $p\geq1$, $\mathcal{U}_{p}=\bigcup\limits_{\lambda\in
\Lambda_{O}}\mathcal{M}_{p}^{\lambda}$.

In $\Lambda_{O}$ we consider the direction%
\begin{equation}
\lambda_{1}\leq\lambda_{2}\Leftrightarrow\mathcal{M}^{\lambda_{1}}%
\leq\mathcal{M}^{\lambda_{2}} \label{ulty2}%
\end{equation}
and, for each $p\geq1$, define%
\begin{equation}
\left\Vert T\right\Vert _{\mathcal{U}_{p}}:=\lim\limits_{\lambda\in\Lambda
_{O}}\left\Vert T\right\Vert _{\mathcal{M}_{p}^{\lambda}}. \label{nor}%
\end{equation}
Note that the above limit exists in view of (\ref{ulty}) and (\ref{ulty2}). It
is not difficult to show that $\left(  \mathcal{U}_{p}\left(  E_{1}%
,...,E_{n};F\right)  ,\left\Vert \cdot\right\Vert _{\mathcal{U}_{p}}\right)  $
is a normed space, for each $E_{1},...,E_{n},F.$ Moreover, $\left(
\mathcal{U}_{p},\left\Vert \cdot\right\Vert _{\mathcal{U}_{p}}\right)  $ is a
normed ideal and one can quickly verify that $\left(  \mathcal{U}%
_{p},\left\Vert \cdot\right\Vert _{\mathcal{U}_{p}}\right)  _{p\geq1}$ is a
desired family. So $\mathcal{U}=(\mathcal{U}_{p})_{p\geq1}\in D$ and
$\mathcal{U}\geq\mathcal{M}$ for all $\mathcal{M}\in D;$ hence Zorn's Lemma
yields that $D$ has a maximal element.
\end{proof}

We also have:

\begin{theorem}
\label{minimal} There exists a desired family of multilinear mappings which is minimal.
\end{theorem}

\begin{proof}
Consider the set $D$ as in the proof of the above theorem and the partial
order%
\[
\mathcal{M}^{\lambda_{2}}\leq\mathcal{M}^{\lambda_{1}}\Leftrightarrow
\mathcal{M}_{p}^{\lambda_{1}}\subseteq\mathcal{M}_{p}^{\lambda_{2}%
}\ \text{and}\ \left\Vert \cdot\right\Vert _{\mathcal{M}_{p}^{\lambda_{2}}%
}\leq\left\Vert \cdot\right\Vert _{\mathcal{M}_{p}^{\lambda_{1}}}\text{ for
all }p\geq1.
\]

Let also $O\subset D$ and $\Lambda_{O}$ be as before. Define%
\[
\mathcal{I}=(\mathcal{I}_{p})_{p\geq1}%
\]
where, for all $p\geq1$, $\mathcal{I}_{p}=\bigcap\limits_{\lambda\in
\Lambda_{O}}^{{}}\mathcal{M}_{p}^{\lambda}$. Note that, for all $p\geq1$, if
$T\in\mathcal{I}_{p}\left(  E_{1},...,E_{n};F\right)  ,$ then%
\begin{equation}
\left\Vert T\right\Vert _{\mathcal{I}_{p}}=\lim_{\lambda\in\Lambda_{O}%
}\left\Vert T\right\Vert _{\mathcal{M}_{p}^{\lambda}} \label{normini}%
\end{equation}
defines a norm in $\mathcal{I}_{p}\left(  E_{1},...,E_{n};F\right)  $. In
fact, from our hypotheses, for each $p\geq1$, the inclusion
\[
inc:\mathcal{L}_{d,p}(E_{1},...,E_{n};F)\rightarrow\mathcal{M}_{p}^{\lambda
}(E_{1},...,E_{n};F)
\]
has norm $\leq1$ for all $\lambda\in\Lambda_{O}.$ So, it follows that
$\left\{  \left\Vert T\right\Vert _{\mathcal{M}_{p}^{\lambda}}:\lambda
\in\Lambda_{O}\right\}  $ is bounded from above by $\left\Vert T\right\Vert
_{d,p}$, and so the limit in $\left(  \ref{normini}\right)  $ exists$.$
Moreover, for all $p\geq1$, the inclusion
\[
inc:\mathcal{M}_{p}^{\lambda}(E_{1},...,E_{n};F)\rightarrow\mathcal{L}%
_{sm,p}(E_{1},...,E_{n};F)
\]
has norm $\leq1$ for every$\ \lambda\in\Lambda_{O}$. Hence
\[
\left\Vert T\right\Vert _{\mathcal{M}_{p}}\geq\left\Vert T\right\Vert
_{sm,p}\geq0\ \text{for all}\ T\in\mathcal{M}_{p}^{\lambda}\left(
E_{1},...,E_{n};F\right)
\]
and so
\[
\left\Vert T\right\Vert _{\mathcal{I}_{p}}=0\text{ if and only if }T=0.
\]
The other properties are easily verified and hence $\left(  \mathcal{I}%
_{p}\left(  E_{1},...,E_{n};F\right)  ,\left\Vert \cdot\right\Vert
_{\mathcal{I}_{p}}\right)  $ is a normed space for all $E_{1},...,E_{n},F$.
The rest of the proof follows the lines of the previous proof.
\end{proof}

\section{Open Problems}

From the previous section two open problems arise:\bigskip

\begin{problem}
Is $(\mathcal{L}_{sm,p})_{p\geq1}$ a desired ideal?(we conjecture that it is
not) If the answer is positive, it will be maximal.
\end{problem}

\begin{problem}
Is $(\mathcal{L}_{ss,p})_{p\geq1}$ a maximal or minimal desired ideal?
\end{problem}

The answer to the next problem seems to be \textquotedblleft
NO\textquotedblright, but to the best of our knowledge, it is an open problem:

\begin{problem}
Is $(\mathcal{L}_{ss,p})_{p\geq1}=(\mathcal{L}_{sm,p})_{p\geq1}$?
\end{problem}

For the next problem that we will propose, we need to define a quite
artificial multilinear version of absolutely summing operators.

Let $n$ be a positive integer, $k\in\{1,...,n\}$ and $p\geq1$. An $n$-linear
operator $T\in\mathcal{L}(E_{1},...,E_{n};F)$ is $k$-absolutely $p$--summing
if there is a constant $C\geq0$ so that%
\begin{equation}
\left(
%TCIMACRO{\dsum \limits_{j_{k}=1}^{\infty}}%
%BeginExpansion
{\displaystyle\sum\limits_{j_{k}=1}^{\infty}}
%EndExpansion
\left\Vert T(x_{j_{1}},...,x_{j_{k}},...,x_{j_{n}})\right\Vert ^{p}\right)
^{1/p}\leq C\left\Vert x_{j_{1}}\right\Vert \ldots\left\Vert \left(  x_{j_{k}%
}\right)  _{j_{k}=1}^{\infty}\right\Vert _{w,p}\ldots\left\Vert x_{j_{n}%
}\right\Vert \label{as}%
\end{equation}
for all $x_{j_{i}}\in E_{i}$ with $i=1,...,k-1,k+1...,n$ and all $\left(
x_{j_{k}}\right)  _{j_{k}=1}^{\infty}$ $\in l_{p}^{w}\left(  E_{k}\right)  .$
In this case we write $T\in\mathcal{L}_{s,p}^{k}(E_{1},...,E_{n};F).$ The
infimum of the $C$ satisfying $\left(  \ref{as}\right)  $ defines a norm for
$\mathcal{L}_{s,p}^{k}(E_{1},...,E_{n};F),$ denoted by $\left\Vert
\cdot\right\Vert _{ks,p}.$ These maps were essentially introduced in
\cite{port} as an \ example of an artificial generalization of absolutely
summing operators.

Now, consider the new class, $\mathcal{L}_{qs,p}$, whose components we will
call quasi-absolutely $p$-summing:
\[
\mathcal{L}_{qs,p}(E_{1},...,E_{n};F):=\bigcap\limits_{k=1}^{n}\mathcal{L}%
_{s,p}^{k}(E_{1},...,E_{n};F).
\]

Defining
\begin{equation}
\left\Vert T\right\Vert _{qs,p}:=\max_{k}\left\Vert T\right\Vert _{ks,p},
\label{nor1}%
\end{equation}
we get a norm for $\mathcal{L}_{qs,p}(E_{1},...,E_{n};F)$ and $(\mathcal{L}%
_{qs,p},\left\Vert .\right\Vert _{qs,p})$ is a Banach ideal.

The ideal $(\mathcal{L}_{qs,p},\left\Vert .\right\Vert _{qs,p})$ is not
interesting since it is the linear ideal in a nonlinear disguise. So, it is
interesting to show that this class is not a desired family. In order to to
this it is necessary to answer the following problem:

\begin{problem}
Is $(\mathcal{L}_{sm,p})_{p\geq1}=(\mathcal{L}_{qs,p})_{p\geq1}?$
\end{problem}

Note that it is plain that $\mathcal{L}_{sm,p}\subset\mathcal{L}_{qs,p}$ and
the inclusion has norm $\leq1$.

Any answer to the above problem will lead to very important conclusions:

- If the answer to the above problem is YES (we conjecture that this is not),
then we have several serious bits of information: (i) the equality is
nontrivial and the result will be interesting by its own; (ii) we conclude
that $(\mathcal{L}_{sm,p})_{p\geq1}$ is a (maximal) desired ideal and (the
more important) we conclude that $(\mathcal{L}_{sm,p})_{p\geq1}$ indeed
possesses very nice properties. For example, besides the inclusion theorem
(which was unknown for this class), since every linear coincidence situation
$\Pi_{p}(E;F)=\mathcal{L}(E;F)$ is naturally extended to $\mathcal{L}%
_{qs,p}(^{n}E;F)=\mathcal{L}(^{n}E;F)$, so we will also have $\mathcal{L}%
_{sm,p}(^{n}E;F)=\mathcal{L}(^{n}E;F)$ and with all this information in hand,
it would be natural to consider $(\mathcal{L}_{sm,p})_{p\geq1}$ as the
"perfect" generalization of $\left(  \Pi_{p}\right)  _{p\geq1}$than
$(\mathcal{L}_{m,p})_{p\geq1}.$

- If the answer to the above problem is NO (which we conjecture), then we
conclude that $(\mathcal{L}_{qs,p})_{p\geq1}$ is not a desired class, a
reasonable situation, since the class $(\mathcal{L}_{qs,p})_{p\geq1}$ is
artificially constructed.

%\begin{acknowledgement}
%The authors expresses their warm thanks to Professor Diestel for the
%encouragement for writing this paper and also for many suggestions/corrections
%that substantially improved its final presentation.
%\end{acknowledgement}

\end{document}